\documentclass[12pt]{amsart}
\usepackage[latin1]{inputenc}
\usepackage{verbatim}
\usepackage{amssymb}
\usepackage{upref}
\usepackage[cmtip,all]{xy}
\usepackage[usenames,dvipsnames]{color}
\usepackage[normalem]{ulem}

\allowdisplaybreaks

\theoremstyle{plain}
\newtheorem{thm}[subsection]{Theorem}
\newtheorem{lem}[subsection]{Lemma}
\newtheorem{cor}[subsection]{Corollary}
\newtheorem{prop}[subsection]{Proposition}
\theoremstyle{definition}
\newtheorem{defn}[subsection]{Definition}

\newtheorem{ex}[subsection]{Example}

\newtheorem{rem}[subsection]{Remark}

\numberwithin{equation}{section}

\allowdisplaybreaks

\newcommand{\thmref}[1]{Theorem~\textup{\ref{#1}}}

\newcommand{\lemref}[1]{Lemma~\textup{\ref{#1}}}
\newcommand{\propref}[1]{Proposition~\textup{\ref{#1}}}

\newcommand{\KK}{\mathcal K}

\newcommand{\OO}{\mathcal O}
\newcommand{\TT}{\mathcal T}
\newcommand{\CC}{\mathcal C}
\newcommand{\DD}{\mathcal D}
\newcommand{\EE}{\mathcal E}
\newcommand{\GG}{\mathcal G}
\newcommand{\HH}{\mathcal H}
\newcommand{\PP}{\mathcal P}
\renewcommand{\AA}{\mathcal A}

\newcommand{\F}{\mathbb F}
\newcommand{\N}{\mathbb N}
\newcommand{\Z}{\mathbb Z}

\newcommand{\variso}{\overset{\simeq}{\longrightarrow}}

\renewcommand{\iff}{\ensuremath{\Leftrightarrow}}
\renewcommand{\bar}{\overline}
\newcommand{\what}{\widehat}
\newcommand{\wilde}{\widetilde}
\newcommand{\inv}{^{-1}}
\newcommand{\<}{\langle}
\renewcommand{\>}{\rangle}

\newcommand{\Chi}{\raisebox{2pt}{\ensuremath{\chi}}}
\renewcommand{\epsilon}{\varepsilon}
\newcommand{\case}{& \text{if }}

\newcommand{\minus}{\setminus}
\newcommand{\act}{\curvearrowright}
\renewcommand{\:}{\colon}
\renewcommand{\subset}{\subseteq}

\renewcommand{\)}{\textup)}

\newcommand{\id}{\text{\textup{id}}}

\DeclareMathOperator{\ad}{Ad}
\DeclareMathOperator{\aut}{Aut}
\DeclareMathOperator*{\spn}{span}
\DeclareMathOperator*{\clspn}{\overline{\spn}}

\newcommand{\arabicnum}{\renewcommand{\labelenumi}{(\arabic{enumi})}}

\newcommand{\midtext}[1]{\quad\text{#1}\quad}
\newcommand{\righttext}[1]{\quad\text{#1 }}

\newcommand{\skpr}{\CC\times_\eta G}

\arabicnum

\begin{document}
\title[Skew products left cancellative small categories]{Skew products of finitely aligned left cancellative small categories and Cuntz-Krieger algebras}
\author{Erik B\'edos}
\address{Institute of Mathematics, University of Oslo, PB 1053 Blindern, 0316 Oslo, Norway}
\email{bedos@math.uio.no}
\author{S. Kaliszewski}
\address{School of Mathematical and Statistical Sciences, Arizona State University, Tempe, AZ 85287}
\email{kaliszewski@asu.edu}
\author{John Quigg}
\address{School of Mathematical and Statistical Sciences, Arizona State University, Tempe, AZ 85287}
\email{quigg@asu.edu}
\date{\today}

\subjclass[2000]{46L05, 46L55}
\keywords{left cancellative small categories, group cocycles, skew products, Toeplitz algebras, Cuntz-Krieger algebras, coactions}

\begin{abstract}
Given a group cocycle on a finitely aligned left cancellative small category (LCSC) we investigate the associated skew product category and its Cuntz-Krieger algebra, which we describe as the crossed product of the Cuntz-Krieger algebra of the original category by an induced coaction of the group.  We use our results to study Cuntz-Krieger algebras arising from free actions of groups on finitely aligned LCSCs, and to construct coactions of groups on Exel-Pardo algebras. Finally we discuss the universal group of a small category and connectedness of skew product categories. 
\end{abstract}

\maketitle

\section{Introduction}\label{intro}
Let $\CC$ be a left cancellative small category (abbreviated as LCSC in the sequel). In a recent work \cite{jacklcsc}, generalizing the previous work of many authors (including himself) dealing with  directed graphs, higher ranks graphs, categories of paths, and  left cancellative monoids, Spielberg 
has shown how to construct certain groupoids from $\CC$ and used these to associate a Toeplitz algebra $\TT(\CC)$ and a Cuntz-Krieger algebra $\OO(\CC)$ to $\CC$. 
For an alternative way of constructing these groupoids, see \cite{OPTight}.
When $\CC$ is finitely aligned, $\TT(\CC)$ and $\OO(\CC)$ may be described by generators and relations in a more tractable way than in the general case (see \cite[Theorems 9.7 and 10.15]{jacklcsc}). We will therefore concentrate our attention on the finitely aligned case in the present paper and use these descriptions of $\TT(\CC)$ and $\OO(\CC)$ as their definitions (cp.\ Section~\ref{lcsc}), as we did in \cite{bkqs}. Actually, to avoid a lot of unnecessary duplication, we will only state our results for $\OO(\CC)$.  

Many interesting LCSCs are not only finitely aligned, but even singly aligned (i.e., the intersection of two principal right ideals is either empty or a principal right ideal). For example, a left cancellative singly aligned monoid corresponds to what is called a right LCM semigroup in some recent articles (e.g.~\cite{BLS2, BLS1}).
Also, the category of finite paths in a directed graph $E$ is singly aligned, and so is the Zappa-Sz{\' e}p product category associated with an Exel-Pardo system $(E, G, \varphi)$, cp.\ \cite{bkqs}.  
More generally, a higher rank graph $(\Lambda, d)$, as defined in \cite{kp:kgraph}, is a LCSC $\Lambda$ equipped with a degree functor $d: \Lambda \to \N^k$, $k \in \N$, satisfying a certain factorization property, and $\Lambda$ is then finitely aligned as a LCSC whenever $(\Lambda, d)$ is finitely aligned in the sense of \cite{raesimsyee}, so in particular when $\Lambda$ is row-finite. 

Given a group cocycle $\eta: \CC \to G$, i.e., a functor from a small category $\CC$ into a group $G$, it is folklore that one may form a new small category $\CC\times_\eta G$, called the \emph{skew product category} associated to $\eta$. At least, such a skew product has previously been introduced and studied when $\CC$ is a groupoid \cite{ren:approach} or a higher rank graph \cite{pqr:cover}. It is not difficult to see that $\CC\times_\eta G$ is a finitely aligned LCSC whenever $\CC$ is (cp.\ Lemma~\ref{skew finitely aligned}), and our main result is then that the cocycle $\eta$ induces a (maximal) coaction $\delta$ of $G$ on $\OO(\CC)$ such that $ \OO(\CC\times_\eta G)$ is isomorphic to the crossed product $\OO(\CC) \rtimes_\delta G$, cp.\ Proposition~\ref{is coaction} and Theorem~\ref{coaction crossed product isomorphism}.  Moreover, there is a natural action of $G$ on $\CC\times_\eta G$, hence on $\OO(\CC\times_\eta G)$,  such that the associated full $C^*$-crossed product  $\OO(\CC\times_\eta G) \rtimes G$ is isomorphic to $\OO(\CC) \otimes \KK(\ell^2(G))$ (where  $\KK(\ell^2(G))$ denotes the compact operators on $\ell^2(G)$), cp.\ Corollary~\ref{dcp}. 

These results generalize similar results from \cite{pqr:cover} in the case of higher rank graphs. 
As in \cite{pqr:cover}, one can draw some interesting consequences from these. For example, we show that if a group $G$ acts freely on a  LCSC $\DD$ in such a way that the quotient category $\DD/G$ (which is left cancellative) is finitely aligned, then $\DD$ is finitely aligned too, and there is a coaction $\delta$ of $G$ on $\OO(\DD/G)$ such that $\OO(\DD)$ is isomorphic to $ \OO(\DD/G) \rtimes_\delta G$, while $\OO(\DD)\rtimes G$ is isomorphic to $\OO(\DD/G)  \otimes \KK(\ell^2(G))$, cp.\ Corollary~\ref{D fin align}. In another direction, if $(E, H, \varphi)$ is an Exel-Pardo system \cite{exelpardo, bkqexelpardo, bkqs}, then we show that any map from the orbit space $E/H$ into a group $G$ induces a coaction of $G$ on the Cuntz-Krieger algebra of the associated Zappa-Sz{\' e}p category, cp.\ Proposition~\ref{ep-coaction}.

The last two sections are devoted to some categorical aspects. It is well known that any small category $\CC$ has a fundamental groupoid  $(\GG(\CC), i)$ having a universal property with respect to groupoid cocycles on $\CC$. Less known is the fact that $\CC$ also has a universal  group $(U(\CC), j)$ 
with respect to group cocycles. If $\CC$ is a monoid, then $\GG(\CC)$ is actually a group, so we have $(U(\CC), j)=(\GG(\CC), i)$ in this case. In the general case,  this fact may be deduced from \cite[Proposition~19, p.~65]{higgins}. We first provide an  elementary proof  of this result, 
cp.\ Proposition~\ref{uni group}. Next, working with $\GG(\CC)$, we describe in Corollary~\ref{uni group connected cat} how $U(\CC)$ may be obtained from the fundamental group of $\CC$ when $\CC$ is connected. Finally, we study and characterize connectedness of  skew product categories. Our main tool is the theory of coverings for small categories. For the ease of the reader we have included a review of this topic, based mainly on \cite{pqr:cover}.

\section{Preliminaries}\label{prelims}

Throughout, $G$ will be a discrete group and $C^*(G)$ will denote its full group $C^*$-algebra. 
The notation $A\otimes B$ for $C^*$-algebras $A,B$ will always denote the minimal tensor product, while $M(A)$ will denote the multiplier algebra of $A$.
By a homomorphism between $*$-algebras, we always mean a $*$-homomorphism.

\subsection{Coactions}

We refer to \cite{maximal, enchilada, eq:full, ngdiscrete, discrete} for coactions of groups.
The \emph{comultiplication} on $C^*(G)$ is the homomorphism
\[
\delta_G\:C^*(G)\to C^*(G)\otimes C^*(G)
\]
given by the integrated form of the unitary homomorphism
\[
g\mapsto g\otimes g\:G\to C^*(G)\otimes C^*(G),
\]
where we identify $G$ with its image in the unitary group of $C^*(G)$.

Let $\delta$ be a coaction of $G$ on a $C^*$-algebra $A$,
so that $\delta\:A\to A\otimes C^*(G)$
is an injective homomorphism
satisfying the \emph{coaction identity},
i.e., the diagram
\begin{equation}\label{coaction}
\xymatrix@C+20pt{
A \ar[r]^-\delta \ar[d]_\delta
&A\otimes C^*(G) \ar[d]^{\delta\otimes\id}
\\
A\otimes C^*(G) \ar[r]_-{\id\otimes\delta_G}
&A\otimes C^*(G)\otimes C^*(G)
}
\end{equation}
commutes, and also satisfying the \emph{nondegeneracy condition for coactions}:
\begin{equation}\label{nondegenerate}
\clspn\{\delta(A)(1\otimes C^*(G))\}=A\otimes C^*(G).
\end{equation}
Nondegeneracy for coactions implies nondegeneracy as a homomorphism into the multiplier algebra $M(A\otimes C^*(G))$.
It is an open problem whether an injective nondegenerate homomorphism satisfying the coaction identity is automatically nondegenerate as a coaction,
although this has been proven for \emph{normal} coactions (see below).

If $(A,\delta)$ is a coaction of $G$ and $g\in G$,
the \emph{$g$-spectral subspace} is
\[
A_g=\{a\in A:\delta(a)=a\otimes g\}.
\]
The spectral subspaces are linearly independent, and nondegeneracy of the coaction is equivalent to the property
\[
A=\clspn\{A_g:g\in G\}.
\]
Moreover, we have
\[
A_gA_h\subset A_{gh}\midtext{and}A_g^*=A_{g\inv}
\righttext{for}g,h\in G,
\]
so the disjoint union $\bigsqcup_{g\in G}A_g$ is a Fell bundle over $G$ with projection map $p\: \bigsqcup_{g\in G}A_g \to G$ 
given by
$p(a) = g$ whenever $a\in A_g$.
As is common in the literature, we sometimes say instead that $\AA=\{A_g\}_{g\in G}$ is a Fell bundle over $G$, and we let $C^*(\AA)$ (resp.~$C_r^*(\AA)$) denote the full (resp.~reduced) $C^*$-algebra of $G$. We refer to \cite{exelbook} for a recent exposition of the theory of Fell bundles over discrete groups.

Given a coaction $(A,\delta)$ of $G$,
we give $A$ the structure of a (left) Banach module over the Fourier-Stieltjes algebra $B(G)$,
using slice maps:
\begin{equation}\label{module}
f\cdot a=(\id\otimes f)\circ \delta(a)\righttext{for}f\in B(G),a\in A.
\end{equation}

The following folklore lemma sometimes aids in the verification that a given map is a coaction:
\begin{lem}\label{faithful}
Let $\delta\:A\to A\otimes C^*(G)$ be a nondegenerate homomorphism
satisfying the coaction identity \eqref{coaction}
and the coaction-non\-de\-gen\-er\-acy condition \eqref{nondegenerate}.
Then $\delta$ is injective, and hence is a coaction.
\end{lem}

\begin{proof}
The hypotheses imply that
the formula
\eqref{module} defines a Banach $B(G)$-module structure on $A$.
Letting $f=1$ be the constant function on $G$ with value 1,
for all $g\in G$ and $a_g\in A_g$
we have
\[
1\cdot a_g=(\id\otimes 1)\circ\delta(a_g)=(\id\otimes 1)(a_g\otimes g)=a_g
\]
because $1(g)=1$ for all $g$,
and by linearity and density we get $(\id\otimes 1)\circ\delta=\id_A$.
\end{proof}

If $B$ is a $C^*$-algebra,
a \emph{covariant representation} of a coaction $(A,\delta)$ in $M(B)$
is a pair $(\pi,\mu)$
of nondegenerate homomorphisms
\[
\xymatrix{
A \ar[r]^-\pi
&M(B)
&c_0(G) \ar[l]_-\mu
}
\]
satisfying the \emph{covariance condition},
i.e., the diagram
\[
\xymatrix@C+60pt{
A \ar[r]^-\delta \ar[d]_\pi
&A\otimes C^*(G) \ar[d]^{\pi\otimes\id}
\\
M(B) \ar[r]_-{\ad (\mu\otimes\id)(w_G)\circ (\cdot\otimes 1)}
&M(B\otimes C^*(G))
}
\]
commutes,
where $w_G$ denotes the unitary element of
\[
M(c_0(G)\otimes C^*(G))=\ell^\infty(G,C^*(G))
\]
given by $w_G(g)=g$ for $g\in G$.

The following  lemma sometimes aids in the verification that a given pair is a covariant representation:
\begin{lem}\label{cov gen lem}
Let $(A,\delta)$ be a coaction of $G$,
let $S\subset G$ and $D_g\subset A_g$ for each $g\in S$,
and let $\pi\:A\to M(B)$ and $\mu\:c_0(G)\to M(B)$ be nondegenerate homomorphisms.
Suppose that $\bigcup_{g\in S}D_g$ generates $A$ as a $C^*$-algebra.
Then $(\pi,\mu)$ is a covariant representation if and only if
\begin{equation}\label{cov gen}
\pi(a_h)\mu(\Chi_g)=\mu(\Chi_{hg})\pi(a_h)\righttext{for all}h\in S,a_h\in D_h,g\in G,
\end{equation}
where $\Chi_g\in c_0(G)$ denotes the characteristic function of $\{g\}$ for $g\in G$.
\end{lem}

\begin{proof}
One direction is well-known (see e.g.~\cite{eq:induced}),  so assume \eqref{cov gen}.
Since $\bigcup_{g\in S}D_g$ generates $A$,
to prove that $(\pi,\mu)$ is a covariant representation
it suffices to check the identity
\begin{equation}\label{cov id}
(\pi\otimes\id)\circ\delta(a_h)=\ad(\mu\otimes\id)(w_G)(\pi(a_h)\otimes 1)
\righttext{for all}h\in S,a_h\in D_h.
\end{equation}

Because the group $G$ is discrete,
we have
\[
w_G=\sum_{g\in G}(\Chi_g\otimes g),
\]
where the sum converges strictly in the multiplier algebra $M(c_0(G)\otimes C^*(G))$.
Thus we get
\begin{align*}
&(\pi\otimes\id)\circ\delta(a_h)(\mu\otimes\id)(w_G)
\\&\quad=\sum_{g\in G}(\pi\otimes\id)\circ\delta(a_h)(\mu(\Chi_g)\otimes g)
\righttext{(strictly convergent)}
\\&\quad=\sum_{g\in G}(\pi(a_h)\otimes h)(\mu(\Chi_g)\otimes g)
\\&\quad=\sum_{g\in G}(\pi(a_h)\mu(\Chi_g)\otimes hg)
\\&\quad=\sum_{g\in G}(\mu(\Chi_{hg})\pi(a_h)\otimes hg)
\\&\quad=\sum_{g\in G}(\mu(\Chi_{g})\pi(a_h)\otimes g)
\righttext{(after $g\mapsto h\inv g$)}
\\&\quad=\sum_{g\in G}(\mu(\Chi_g)\otimes g)(\pi(a_h)\otimes 1)
\\&\quad=(\mu\otimes\id)(w_G)(\pi(a_h)\otimes 1),
\end{align*}
and so \eqref{cov id} holds.
\end{proof}

Given a coaction $(A,\delta)$ of $G$,
we write $A\rtimes_\delta G$
for the crossed product $C^*$-algebra,
and $(j_A,j_G)$ for the universal covariant representation of $(A,\delta)$ in $M(A\rtimes_\delta G)$,
meaning that for every covariant representation
$(\pi,\mu)$ of $(A,\delta)$ in $M(B)$ there is a unique nondegenerate homomorphism $\pi\times\mu\:A\rtimes_\delta G\to M(B)$ making the diagram
\[
\xymatrix{
A \ar[r]^-{j_A} \ar[dr]_\pi
&M(A\rtimes_\delta G) \ar@{-->}[d]^(.4){\pi\times\mu}_(.4){!}
&c_0(G) \ar[l]_-{j_G} \ar[dl]^\mu
\\
&M(B)
}
\]
commute.
Then
we have
\[A\rtimes_\delta G
=\clspn\big\{j_A(a_g)j_G(\Chi_h):g,h\in G,a_g\in A_g\big\} \]
and we can describe the crossed product as the closed linear span of pairs
\[
(a_g,h)\in A_g\times G,
\]
with operations given by
\begin{equation}\label{eqcalculus}
\begin{split}
(a_g,h)(b_k,\ell)&=\begin{cases}
(a_gb_k,\ell)
\case
h=k\ell,\\
\ 0 
& \text{otherwise;}
\end{cases}
\\
(a_g,h)^*&=(a_g^*,gh).
\end{split}
\end{equation}

Indeed, we only need to identify
\[
(a_g,h)=
j_A(a_g)j_G(\Chi_h).
\]

If $(A,\delta)$ is a coaction of $G$,
there is a \emph{dual action} $\what\delta$ of $G$ on $A\rtimes_\delta G$
given on generators by
\[
\what\delta_g(a_h,k)=
(a_h,kg\inv).
\]
Note that this is obtained from the usual dual action,
which is trivial on $j_A(A)$ and for which the homomorphism $j_G\:c_0(G)\to M(A\rtimes_\delta G)$ is equivariant for right translation on $c_0(G)$:
\[
\what\delta_g(a_h,k) =
\what\delta_g(j_A(a_h)j_G(\Chi_k)) = j_A(a_h)j_G(\Chi_{kg\inv})
=(a_h,kg\inv).
\]

There is a canonical surjective homomorphism
\[
\Phi\:A\rtimes_\delta G\rtimes_{\what\delta} G\to A\otimes \KK(\ell^2(G)),
\]
and
$\delta$ is \emph{maximal} if $\Phi$ is an isomorphism,
and
\emph{normal} if $\Phi$ factors through an isomorphism on the reduced crossed product:
\[
\xymatrix{
A\rtimes_\delta G \rtimes_{\what\delta} G \ar[r]^-\Phi \ar[d]_\Lambda
&A\otimes \KK
\\
A\rtimes_\delta G \rtimes_{\what\delta,r} G, \ar @{-->}[ur]_\simeq
}
\]
where $\Lambda$ is the regular representation onto the reduced crossed product by the dual action.
However, since $G$ is discrete,
to test for maximality we can use the following characterization:
letting $\AA=\{A_g\}_{g\in G}$ be the Fell bundle associated to the coaction $\delta$,
the inclusion maps $A_g\hookrightarrow A$
extend to a
surjection
\[
C^*(\AA)\to A,
\]
and $\delta$ is maximal if and only if this surjection
is injective \cite{maximal}.

\subsection{LCSC} \label{lcsc}

We refer to \cite{bkqs, jacklcsc} for left cancellative small categories and the $C^*$-algebras that may be associated to these.
For the special case of 
``categories of paths'', which includes higher-rank graphs, see \cite{jackpaths}.
If $\CC$ is a small category, we refer to an object as a \emph{vertex}, and denote by $\CC^0$ the set of vertices in $\CC$. We identify the objects with the identity morphisms.
A \emph{small category} can be defined as a set $\CC$, a subset $\CC^0$, two maps $r,s\:\CC\to\CC^0$ (called the range map and the source map, respectively), and a partially-defined multiplication
$(\alpha,\beta)\mapsto \alpha\beta$,
defined if and only if $s(\alpha)=r(\beta)$,
such that
for all $\alpha,\beta,\gamma\in\CC$ with $s(\alpha)=r(\beta)$ and $s(\beta)=r(\gamma)$,
\begin{enumerate}
\item $r(\alpha\beta)=r(\alpha)$ and $s(\alpha\beta)=s(\beta)$;
\item $\alpha(\beta\gamma)=(\alpha\beta)\gamma$;
\item $r(v)=s(v)=v$ for all $v\in \CC^0$;
\item $r(\alpha)\alpha=\alpha s(\alpha)=\alpha$.
\end{enumerate}
When writing $\alpha\beta$ for $\alpha, \beta\in \CC$, we often tacitly assume that  the product $\alpha\beta$ is defined. If $E,F\subset\CC$ we write $EF=\{\alpha\beta:\alpha\in E,\beta\in F\}$.

A \emph{left-cancellative small category} (LCSC) is a small category $\CC$
 such that for $\alpha,\beta,\gamma\in\CC$,
if $\alpha\beta=\alpha\gamma$ then $\beta=\gamma$.
From now on, $\CC$ will denote a LCSC.

The equivalence relation $\sim$ on $\CC$ is defined by
$\alpha\sim\beta$ if there is an invertible $\gamma\in\CC$ such that $\alpha=\beta\gamma$.

$\CC$ is \emph{finitely aligned} if
for all $\alpha,\beta\in\CC$ there is a (possibly empty) finite subset $F\subset \CC$ such that
\[
\alpha\CC\cap \beta\CC= F\CC\left(=\bigcup_{\gamma\in F}\gamma\CC\right).
\]

A subset $F\subset\CC$ is \emph{independent} if
for all $\alpha,\beta\in F$ with $\alpha\in \beta\CC$ we have $\alpha=\beta$.
If $\CC$ is finitely aligned and $F\subset\CC$ is finite,
we write $\bigvee F$ for any finite independent subset $L\subset\CC$ such that
\[
\bigcap_{\alpha\in F}\alpha\CC=L\CC.
\]
Note that $\bigvee F$ is only unique up to equivalence:
if $L'$ is another set with the same properties as $L$ then for all $\beta\in L$ there exists $\gamma\in L'$ such that $\beta\sim\gamma$,
and symmetrically for all $\gamma\in L'$ there exists $\beta\in L$ such that $\gamma\sim\beta$. 
When $F=\{\alpha,\beta\}$ we write
$\alpha\vee \beta$ instead of $\bigvee F$.

If $v\in\CC^0$ and $F\subset v\CC$, then $F$ is \emph{exhaustive} at $v$ if for all $\alpha\in v\CC$ there exists $\beta\in F$ such that $\alpha\CC\cap \beta\CC\ne\varnothing$.

A \emph{representation} of $\CC$ in a $C^*$-algebra $B$ is a map $T\:\CC\to B$ such that
for all $\alpha,\beta\in\CC$,
\begin{enumerate}
\item $T_\alpha^*T_\alpha=T_{s(\alpha)}$;
\item $T_\alpha T_\beta=T_{\alpha\beta}$ if $s(\alpha)=r(\beta)$;
\item $T_\alpha T_\alpha^*T_\beta T_\beta^*
=\bigvee_{\gamma\in \alpha\vee\beta} T_\gamma T_\gamma^*$,
\end{enumerate}
and a representation $T$ is \emph{covariant} if
\begin{enumerate}
\setcounter{enumi}{3}
\item $T_v=\bigvee_{\alpha\in F} T_\alpha T_\alpha^*$ for all $v\in\CC^0$ and every finite exhaustive set $F$ at $v$.
\end{enumerate}
We recall that if $T$ is a representation then $T_v$ is a projection in $B$ for every $v \in \CC^0$, and $T_\alpha$ is a partial isometry in $B$ for every $\alpha \in \CC$, so $T_\alpha T_\alpha^*$ is a projection in $B$ for every $\alpha \in \CC$. We also recall that the joins $\bigvee_{\gamma\in \alpha\vee \beta}T_\gamma T_\gamma^*$ and $\bigvee_{\alpha\in F}T_\alpha T_\alpha^*$ are a priori defined as projections in $B^{**}$, and that, by convention, the join over an empty index set is defined to be zero. Thus, if $v, w \in \CC^0$ and $v\neq w$, then $T_v T_w = 0$, i.e., the projections $T_v$ and $T_w$ are orthogonal to each other. 

Moreover,
 if $T$ is a representation
then:
\begin{itemize}
\item $T_\alpha^* T_\beta=0$ if $s(\alpha)\ne s(\beta)$;
\item if $\alpha$ is invertible then $T_{\alpha\inv}=T_\alpha^*$
and $T_\alpha T_\alpha^*=T_{r(\alpha)}$;
\item if $\beta\sim\gamma$ then $T_\beta T_\beta^*=T_\gamma T_\gamma^*$.
\end{itemize}
The \emph{Toeplitz algebra} $\TT(\CC)$ is generated by a universal representation of $\CC$,
and the \emph{Cuntz-Krieger algebra} $\OO(\CC)$ is generated by a universal covariant representation of $\CC$.
We are more interested in $\OO(\CC)$ than $\TT(\CC)$,
and we write $t\:\CC\to \OO(\CC)$ for the universal covariant representation.
The universal property means that for every covariant representation $T\:\CC\to B$
there is unique homomorphism
$\phi\:\OO(\CC)\to B$,
called the \emph{integrated form} of $T$,
such that $\phi\circ t=T$.

As the reader will easily realize, all our results concerning Cuntz-Krieger algebras of finitely aligned LCSCs in the present paper carry over to the associated Toeplitz algebras, with simpler proofs.

\section{Cocycles and skew products}
\label{sec:skew}

In this section we let $\CC$ be a 
LCSC
and $G$ a discrete group. We consider  
a $G$-valued \emph{cocycle} $\eta$ on $\CC$, that is, 
a functor $\eta\:\CC\to G$.

\begin{prop}\label{is coaction}
Assume that $\CC$ is finitely aligned. 
Let $t\:\CC\to \OO(\CC)$ be the universal representation,
and define $T\:\CC\to \OO(\CC)\otimes C^*(G)$ by
\[
T_\alpha=t_\alpha\otimes \eta(\alpha).
\]
Then $T$ is a covariant representation of $\CC$,
and its integrated form $\delta\:\OO(\CC)\to \OO(\CC)\otimes C^*(G)$ is a coaction.
\end{prop}

\begin{proof}
We first check the axioms for a 
covariant
representation:

(1)
For all $\alpha\in\CC$,
\begin{align*}
T_\alpha^*T_\alpha
&=(t_\alpha\otimes \eta(\alpha))^*(t_\alpha\otimes \eta(\alpha))
\\&=(t_\alpha^*\otimes \eta(\alpha)\inv)(t_\alpha\otimes \eta(\alpha))
\\&=t_\alpha^*t_\alpha\otimes 1
\\&=t_{s(\alpha)}\otimes 1
\\&=T_{s(\alpha)},
\end{align*}

(2)
if $s(\alpha)=r(\beta)$ then
\begin{align*}
T_\alpha T_\beta
&=(t_{\alpha}\otimes \eta(\alpha))(t_{\beta}\otimes \eta(\beta))
\\&=t_\alpha t_\beta\otimes \eta(\alpha)\eta(\beta)
\\&=t_{\alpha\beta}\otimes \eta(\alpha\beta)
\\&=T_{\alpha\beta},
\end{align*}

(3)
for all $\alpha,\beta\in\CC$,
\begin{align*}
T_\alpha T_\alpha^* T_\beta T_\beta^*
&=(t_\alpha\otimes \eta(\alpha))
(t_\alpha\otimes \eta(\alpha))^*
(t_\beta\otimes \eta(\beta))
(t_\beta\otimes \eta(\beta))^*
\\&=(t_\alpha\otimes \eta(\alpha))
(t_\alpha^*\otimes \eta(\alpha)\inv)
(t_\beta\otimes \eta(\beta))
(t_\beta^*\otimes \eta(\beta)\inv)
\\&=t_\alpha t_\alpha^* t_\beta t_\beta^*\otimes 1
\\&=\left(\bigvee_{\gamma\in \alpha\vee \beta} t_\gamma t_\gamma^*\right)\otimes 1
\\&=\bigvee_{\gamma\in \alpha\vee \beta}
(t_\gamma t_\gamma^*\otimes 1)
\\&=\bigvee_{\gamma\in \alpha\vee \beta}
\bigl(t_\gamma t_\gamma^*\otimes \eta(\gamma)\eta(\gamma)\inv\bigr)
\\&=\bigvee_{\gamma\in \alpha\vee \beta}T_\gamma T_\gamma^*,
\end{align*}
and

(4)
for any finite exhaustive set $F$ at $v\in\CC^0$,
\begin{align*}
T_v
&=t_v\otimes \eta(v)
\\&=t_v\otimes 1
\\&=\left(\bigvee_{\alpha\in F} t_\alpha t_\alpha^*\right)\otimes 1
\\&=\bigvee_{\alpha\in F}\bigl(t_\alpha t_\alpha^*\otimes 1\bigr)
\\&=\bigvee_{\alpha\in F}\bigl(t_\alpha t_\alpha^*\otimes \eta(\alpha)\eta(\alpha)\inv\bigr)
\\&=\bigvee_{\alpha\in F}T_\alpha T_\alpha^*.
\end{align*}

Thus we have a homomorphism
$\delta\:\OO(\CC)\to \OO(\CC)\otimes C^*(G)$
that is the integrated form of $T$.

The coaction identity is easily checked on generators:
\begin{align*}
(\delta\otimes\id)\circ\delta(t_\alpha)
&=t_\alpha\otimes \eta(\alpha)\otimes \eta(\alpha)
\\&=(\id\otimes \delta_G)(t_\alpha\otimes \eta(\alpha))
\\&=(\id\otimes \delta_G)\circ \delta(t_\alpha).
\end{align*}

This homomorphism $\delta$
satisfies the coaction-nondegeneracy condition \eqref{nondegenerate} because
every elementary tensor $a\otimes c\in \OO(\CC)\otimes C^*(G)$
is in the $C^*$-subalgebra generated by
elementary tensors of the form
$t_\alpha\otimes \eta(\alpha)g$
for $\alpha\in\CC$ and $g\in G$.
It now follows from \lemref{faithful}
that $\delta$ is faithful, and hence is a coaction of $G$ on $\OO(\CC)$.
\end{proof}

\cite[Theorem~7.1]{pqr:cover}
proves that,
given a $G$-valued cocycle on a $k$-graph,
the associated coaction $\delta$ is both maximal and normal.
In the 
context of cocycles on finitely aligned LCSCs,
it is natural to ask whether these properties of $\delta$ continue to hold, and we answer this in the following theorem.

In preparation,
recall
the notation
$P_0=\{t_\alpha t_\alpha^*:\alpha\in\CC\}$
and 
$P=\{p_1\cdots p_n: n\in \N, p_1, \ldots, p_n\in P_0\}$
from \cite[Section~6]{jackpaths},
and note that the proof of \cite[Proposition~6.7]{jackpaths} implies that
\[
\OO(\CC)=\clspn\{t_\alpha t_\beta^*\,q:\alpha,\beta\in\CC, q\in P\}.
\]

\begin{thm}\label{maximal}
Assume that $\CC$ is finitely aligned. Let
$\delta$ be the coaction of $G$ on $\OO(\CC)$ from \propref{is coaction},
and let $\AA$ be the associated Fell bundle.
Then the associated surjection $\pi\:C^*(\AA)\to \OO(\CC)$ is an isomorphism.
In particular,
$\delta$ 
is maximal.
However,
$\delta$ may fail to be normal.
Finally, the fibres of the Fell bundle $\AA$ are given by
\begin{equation}\label{spectral}
A_g=\clspn\{t_\alpha t_\beta^*\,q:\eta(\alpha)\eta(\beta)\inv=g,q\in P\}.
\end{equation}
\end{thm}

\begin{proof}
For the 
first assertion,
we must show that $\pi$ is injective.
Define a map $T\:\CC\to C^*(\AA)$ by
\[
T_\alpha=t_\alpha\righttext{(regarded now as an element of $C^*(\AA)$).}
\]
This definition of $T$ might appear confusing at first,
but most of the point of this proof is to avoid conflating things that are really in different places.
We are accustomed to regarding the generators $t_\alpha$ as elements of $\OO(\CC)$,
and now we need to think of them alternatively as elements of $C^*(\AA)$.
The algebraic structure of $C^*(\AA)$ guarantees that $T$ is a covariant representation.
Thus its integrated form is a homomorphism $\rho\:\OO(\CC)\to C^*(\AA)$.

Note that
\[
\pi(T_\alpha)=t_\alpha\righttext{for all}\alpha\in\CC.
\]
Thus $\rho\circ\pi$ is an endomorphism of $C^*(\AA)$ that agrees with the identity automorphism $\id$
on the (image in $C^*(\AA)$ of the) Fell bundle $\AA$,
and hence $\rho\circ\pi=\id$.
Therefore $\pi$ is also injective, and hence is an isomorphism.

To see the possible failure of normality,
just note that the group $G$ itself is a finitely aligned LCSC,
and we can take the cocycle $\eta\:G\to G$ to be the identity map,
in which case
the associated coaction is $\delta_G$.
As is well-known, $\delta_G$ is normal if and only if $G$ is amenable.

Finally, for the equality~\eqref{spectral},
fix $g\in G$, and
let $R$ denote the right-hand side.
Note first of all that by definition
\[
A_g=\{a\in \OO(\CC):\delta(a)=a\otimes g\},
\]
so because $\delta(q) = q\otimes 1$ for all $q\in P$,
we have $R\subset A_g$.

For the opposite containment,
we argue as in the proof of \cite[Lemma~7.9]{pqr:cover}:
let $\Chi_g\in B(G)$ be the characteristic function of $\{g\}$.
Then
by the basic theory of coactions,
there is an idempotent bounded linear surjection
\[
E_g=(\id\otimes \Chi_g)\circ\delta\: \OO(\CC)\to A_g.
\]
Any $a\in \OO(\CC)$ can be approximated by a linear combination
\[\sum_{i=1}^n c_it_{\alpha_i} t_{\beta_i}^*q_i\]
with $\alpha_i,\beta_i\in \CC$ and $q_i\in P$,
and hence, if $a\in A_g$, then
\begin{align*}
a
&=E_g(a)
\approx E_g\left(\sum_{i=1}^n c_it_{\alpha_i}t_{\beta_i}^*q_i\right)
=\sum_{i=1}^n c_iE_g(t_{\alpha_i}t_{\beta_i}^*q_i)
\\&=\sum_{\substack{ i=1, \ldots, n,\\ \eta(\alpha_i)\eta(\beta_i)\inv=g}} c_it_{\alpha_i}t_{\beta_i}^*q_i\ \in R.
\qedhere\end{align*}
\end{proof}

Combined with the existing theory for Fell bundles over discrete groups, Theorem~\ref{maximal} provides a useful tool for investigating certain properties of $\OO(\CC)$. We illustrate this by giving a set of conditions ensuring that $\OO(\CC)$ is nuclear.
We first recall from \cite[Definition~20.4]{exelbook} that a Fell bundle $\mathcal{B}=\{B_g\}_{g\in G}$ over $G$ is said to have \emph{the approximation property} if there exists a net $\{a_i\}_{i\in I} $ of finitely supported functions from $G$ into  $B_e$ satisfying \[\sup_{i\in I} \| \sum_{g\in G} a_i(g)^* a_i(g)\| < \infty\]
and
\[ \lim_i \sum_{h\in G} a_i(gh)^*ba_i(h) = b \quad \text{for all } b \text{ in each $B_g$}.\] 
We also recall that this property is satisfied whenever $G$ is amenable (cp.\ \cite[Theorem~2.7]{exelbook}). 
\begin{prop} Assume that $\CC$ is finitely aligned, and let $\AA$ be the Fell bundle over $G$ as in Theorem~\ref{maximal}, so that \[A_e=\clspn\{t_\alpha t_\beta^*\,q:\eta(\alpha)=\eta(\beta), q\in P\},\] 
where $e$ denotes the unit of $G$. Assume that $\AA$ has the approximation property and that $A_e$ is nuclear. Then $\OO(\CC)$ is nuclear.
\end{prop}
\begin{proof} This follows by combining Theorem~\ref{maximal} with \cite[Proposition~25.10]{exelbook}. 
\end{proof} 

\begin{rem} Kwa\'{s}niewski and Meyer have shown in \cite[Theorem 5.8]{KM:Fell}
 how strong pure infiniteness of the reduced $C^*$-algebra of a Fell bundle over a group can be deduced from certain properties of the Fell bundle. Thus,  if $G$ is amenable, one may use this result  in combination with Theorem~\ref{maximal} to obtain conditions on the Fell bundle $\AA$ associated to $\eta$ ensuring that $\OO(\CC)$ will be strongly purely infinite. Similarly, in the case where $\eta:\CC \to \Z$, one may invoke   \cite[Theorem~9.7]{KM:Fell} to state conditions on $\AA$ sufficient for $\OO(\CC)$ to be simple. 
 It should here be noted that Ortega and Pardo \cite{OPTight} have recently given sufficient conditions for $\OO(\CC)$ to be simple, and even have a result explicitly for Zappa-Sz\'ep products.
 \end{rem}

We next introduce the skew product category associated to a cocycle on a LCSC. This construction is well-known, at least in the case of groupoids and higher-rank graphs, although some authors use different conventions than ours. 

\begin{defn}\label{skew def}
The \emph{skew product}
$\skpr$
of $\CC$ by a cocycle $\eta\:\CC\to G$
is the set
$\CC\times G$
with partially-defined multiplication
\[
(\alpha,g)(\beta,h)=(\alpha\beta,h)\righttext{if} s(\alpha)=r(\beta)\text{ and }g=\eta(\beta)h.
\]
\end{defn}

\begin{lem}\label{skew finitely aligned}
With the above multiplication, $\skpr$ is a 
LCSC,
with range and source maps given by
\begin{align*}
s(\alpha,g)&=
(s(\alpha),g),\\
r(\alpha,g)&=
(r(\alpha),\eta(\alpha)g),
\end{align*}
and vertex set given by
$(\skpr)^0=\CC^0\times G$.
Moreover, $\skpr$ is finitely aligned if $\CC$ is.
\end{lem}

\begin{proof}
For this proof let $\DD=\skpr$.
We first check that $\DD$ is a small category.
By construction, $\DD$ is a set.
We have maps $s,r\:\DD\to \DD^0$
such that $(\alpha,g)(\beta,h)$ is defined if and only if $s(\alpha,g)=r(\beta,h)$.
We check the axioms:

(1)
if $s(\alpha,g)=r(\beta,h)$
then
\begin{align*}
r\bigl((\alpha,g)(\beta,h)\bigr)
&=
r(\alpha\beta,h)
\\&=
(r(\alpha\beta),\eta(\alpha\beta)h)
\\&=(r(\alpha),\eta(\alpha) \eta(\beta)h)
\\&=(r(\alpha),\eta(\alpha) g)\righttext{\quad (since $\eta(\alpha)h=g$)}
\\&=r(\alpha,g)
\end{align*}
\quad \quad \, \,  and
\begin{align*}
s\bigl((\alpha,g)(\beta,h)\bigr)
&=
s(\alpha\beta,h)
\\&=
(s(\alpha\beta),h)
\\&=
(s(\beta), h)
\\&=s(\beta,h),
\end{align*}

(2)
if $s(\alpha,g)=r(\beta,h)$ and $s(\beta,h)=r(\gamma,k)$
then
\begin{align*}
(\alpha,g)\bigl((\beta,h)(\gamma,k)\bigr)
&=
(\alpha,g)(\beta\gamma,k)
\\&=
(\alpha\beta\gamma,k)
\\&\hspace{.1in}\text{(since $g= \eta(\beta)h = \eta(\beta)\eta(\gamma)k = \eta(\beta\gamma) k$)}
\\&=
(\alpha\beta,h)(\gamma,k)
\\&=\bigl((\alpha,g)(\beta,h)\bigr)(\gamma,k),
\end{align*}

(3)
for all $v\in\CC^0,g\in G$ we have
\begin{align*}
r(v,g)
&=
(r(v),\eta(v)g)
\\&=
(v,g) \righttext{(since $\eta(v)=1$)}
\\&=
(s(v), g)
\\&=s(v,g),
\end{align*}
and

(4)
for all $(\alpha,g)\in \DD$
we have
\begin{align*}
r(\alpha,g)(\alpha,g)
&=
(r(\alpha),\eta(\alpha)g)(\alpha,g)
\\&=(r(\alpha)\alpha,g)
\\&=(\alpha,g)
\\&=(\alpha s(\alpha),g)
\\&=
(\alpha,g)(s(\alpha), g)
\\&\hspace{.5in}\text{(since $\eta(s(\alpha)) g=1g=g$)}
\\&=(\alpha,g)s(\alpha,g).
\end{align*}

Now we check left cancellation:
let $(\alpha,g),(\beta,h),(\gamma,k)\in \DD$
with $s(\alpha,g)=r(\beta,h)=r(\gamma,k)$,
and assume that
\[
(\alpha,g)(\beta,h)=(\alpha,g)(\gamma,k).
\]
Then
\[
g=\eta(\beta)h=\eta(\gamma)k
\]
and
\[
\alpha\beta=\alpha\gamma,
\]
so $\beta=\gamma$ since $\CC$ is left cancellative, and $h=\eta(\beta)\inv g = \eta(\gamma)\inv g = k$.

For the last part, suppose that $\CC$ is finitely aligned and consider $(\alpha,g),(\beta,h)\in \DD$.

We first show that
$(\alpha,g)\DD\cap (\beta,h)\DD$
is equal to
\[
\bigl\{\bigl(\gamma, \eta(\gamma)\inv \eta(\alpha)g\bigr):
\gamma\in\alpha\CC\cap \beta\CC)\bigr\}
\]
if
$\alpha\CC\cap \beta\CC\ne\varnothing$
and
$\eta(\alpha)g = \eta(\beta)h$,
and empty otherwise.

If $(\gamma,k)\in (\alpha,g)\DD\cap (\beta,h)\DD$,
then we have
\[
(\gamma,k)=(\alpha,g)(\mu,p)=(\beta,h)(\nu,q)
\]
for some $(\mu,p),(\nu,q)\in \DD$.
Then
\[
g=\eta(\mu)p\midtext{and} h= \eta(\nu)q,
\]
and
\[
(\gamma, k)=
(\alpha\mu,p)=(\beta\nu,q).
\]
This forces 
$k=p=q$, and
\[
\gamma=\alpha\mu=\beta\nu
\in \alpha\CC\cap \beta\CC,
\]
hence
\[\eta(\alpha)g = \eta(\alpha)\eta(\mu)p = \eta(\alpha\mu)p =  \eta(\beta\nu)q = \eta(\beta)\eta(\nu)q = \eta(\beta)h.\]
Thus no such $(\gamma,k)$ exists if 
$\alpha\CC\cap \beta\CC=\varnothing$ or $\eta(\alpha)g\ne \eta(\beta)h$.
Moreover, if 
$\alpha\CC\cap \beta\CC\ne\varnothing$ and $\eta(\alpha)g= \eta(\beta)h$
then
$\gamma\in \alpha\CC\cap \beta\CC$
and
\[
\eta(\gamma)\inv \eta(\alpha)g =\eta(\mu)\inv \eta(\alpha)\inv \eta(\alpha)g = p=k.
\]
On the other hand, if $\eta(\alpha)g= \eta(\beta)h, \gamma\in\alpha\CC\cap \beta\CC$ and $k=\eta(\gamma)\inv \eta(\alpha)g$
then 
\[
\gamma=\alpha\mu=\beta\nu
\]
for some $\mu,\nu\in\CC$,
so that
\begin{align*}
(\gamma,k)
&=(\alpha\mu, \eta(\gamma)\inv \eta(\alpha)g)
\\&=(\alpha\mu, \eta(\mu)\inv g)
\\&=(\alpha,g)(\mu,\eta(\mu)\inv g)
\end{align*}
and
\begin{align*}
(\gamma,k)
&=(\beta\nu, \eta(\gamma)\inv \eta(\beta)h)
\\&=(\beta\nu, \eta(\nu)\inv h)
\\&=(\beta, h)(\nu, \eta(\nu)\inv h),
\end{align*}
hence $(\gamma,k)\in (\alpha,g)\DD\cap (\beta,g)\DD.$

Now assume that $\alpha\CC\cap \beta\CC\ne\varnothing$ and $\eta(\alpha)g = \eta(\beta)h$. Let
 $F$ be a nonempty finite subset of $ \CC$ such that $\alpha\CC\cap \beta\CC = \cup_{\delta\in F} \,\delta\CC$.
 We will show that
 \[ (\alpha,g)\DD\cap (\beta,g)\DD = \bigcup_{\delta \in F} \bigl(\delta, \eta(\delta)\inv\eta(\alpha)g\bigr)\DD,\]
 which will prove that $\DD$ is finitely aligned. 

Consider first $(\gamma, k) \in (\alpha,g)\DD\cap (\beta,g)\DD$. Then, as we have shown above, $\gamma \in \alpha\CC\cap \beta\CC$ and $k=\eta(\gamma)\inv\eta(\alpha)g$. We can now pick $\delta \in F$ such that $\gamma \in \delta\CC$, i.e., $\gamma = \delta x $ for some $x \in \CC$. This gives that $k = \eta(x)\inv\eta(\delta)\inv\eta(\alpha)g$, hence $
\eta(\delta)\inv\eta(\alpha)g = \eta(x)k$, so we get
\[ (\gamma, k) = (\delta x, k) = \bigl(\delta, \eta(\delta)\inv\eta(\alpha)g\bigr) (x, k) \in \bigl(\delta, \eta(\delta)\inv\eta(\alpha)g\bigr)\DD.\]

Conversely, assume that $(\gamma, k) \in\bigl(\delta, \eta(\delta)\inv\eta(\alpha)g\bigr)\DD$ for some $\delta \in F$, therefore we have $(\gamma, k) = \bigl(\delta, \eta(\delta)\inv\eta(\alpha)g\bigr)(x, a)$ for some $(x, a) \in \DD$ such that $\eta(\delta)\inv\eta(\alpha)g=\eta(x) a$. This implies that 
\[ \gamma = \delta x \in \delta\CC \subseteq \alpha\CC\cap \beta\CC \quad \text{and}\]
\[k= a = \eta(x)\inv\eta(\delta)\inv\eta(\alpha)g= \eta(\gamma)\inv\eta(\alpha)g.\]
It follows that $(\gamma, k) \in (\alpha,g)\DD\cap (\beta,g)\DD$, as wanted. 

As a bonus, we remark that one easily proves that $\bigcup_{\delta \in F} \bigl(\delta, \eta(\delta)\inv\eta(\alpha)g\bigr)$ is independent if $F$ is independent, so one gets that
\begin{equation}\label{vee-skew}  (\alpha,g)\vee (\beta,h) = \bigcup_{\delta\in \alpha\vee\beta} \bigl(\delta, \eta(\delta)\inv\eta(\alpha)g\bigr)
\end{equation}
whenever $\eta(\alpha)g=\eta(\beta)h.$
\end{proof}

\begin{lem}\label{act skrew}
For each $g\in G$ and $ (\beta, h)\in \skpr$, set
\[ g\cdot (\beta, h) = 
(\beta, hg\inv) \in \skpr.\]
This gives an action of $G$ on $\skpr$ by category automorphisms. 
\end{lem} 
\begin{proof} This is straight-forward.
For example, let $g\in G$, and assume that $\big((\beta, h), (\gamma, k)\big)$ 
is a composable pair in  $\skpr$. 
Then we have
\[  g\cdot \big((\beta, h)(\gamma, k)\big) =  
g\cdot (\beta\gamma, k)
= (\beta\gamma, kg\inv).
\]
Moreover, since 
$ hg = \eta(\gamma)kg$, the pair 
$((\beta, hg),(\gamma, kg))$ is also 
 composable  in $\skpr$, and we get
\[  \big(g\cdot (\beta, h)\big)\big(g\cdot(\gamma, k)\big) =  
(\beta, hg\inv)(\gamma, kg\inv) =
(\beta\gamma, kg\inv).
\]
Thus it follows that $g\cdot \big((\beta, h)(\gamma, k)\big)=\big(g\cdot (\beta, h)\big)\big(g\cdot(\gamma, k)\big)$. 
\end{proof}
Now suppose that $\CC$ is finitely aligned. 
Using the universal property of $\OO(\skpr)$ we get that
the action of $G$ on $\skpr$ induces 
an action $\gamma\:G\curvearrowright \OO(\skpr)$ determined by
\[
\gamma_g(t_{(\beta, h)})=t_{g\cdot (\beta, h)} = 
t_{(\beta, hg\inv)}
\]
for all $g, h\in G$ and $\beta \in \CC$.

On the other hand, we have the dual action $\what\delta\:G\curvearrowright \OO(\CC)\rtimes_\delta G$ characterized by
\[
\what\delta_g(a_h,k)=(a_h,kg\inv)
\]
for all $g,h,k\in G$ and $a_h\in A_h$.

\begin{thm}\label{coaction crossed product isomorphism}
Assume that $\CC$ is finitely aligned. Define $T\:\CC\times_\eta G \to \OO(\CC)\rtimes_\delta G$ by
\[
T_{(\alpha,g)}=(t_\alpha,g).
\]
Then $T$ is a covariant representation of the skew product $\CC\times_\eta G$,
and its integrated form $\theta$ is an isomorphism
\[
\OO(\skpr)\variso \OO(\CC)\rtimes_\delta G.
\]
Moreover, this isomorphism is $\gamma-\what\delta$ equivariant.
\end{thm}

\begin{proof}
For this proof let $\DD=\skpr$.
We check the axioms for $T$ to be a covariant representation:

(1)
For all $(\alpha,g)\in \DD$,
\begin{align*}
T_{(\alpha,g)}^* T_{(\alpha,g)}
&=(t_\alpha,g)^*(t_\alpha,g)
\\&=(t_\alpha^*,\eta(\alpha) g)(t_\alpha,g)
\\&=(t_\alpha^* t_\alpha,g)
\\&=(t_{s(\alpha)},g)
\\&=T_{(s(\alpha),g)}
\\&=T_{s(\alpha,g)},
\end{align*}

(2)
if $s(\alpha,g)=r(\beta,h)$, then
\begin{align*}
T_{(\alpha,g)} T_{(\beta,h)}
&=(t_\alpha,g)(t_\beta,h)
\\&=(t_\alpha t_\beta,h)
\\&=(t_{\alpha\beta},h)
\\&=T_{(\alpha\beta,h)}
\\&=T_{(\alpha,g)(\beta,h)},
\end{align*}
and

(3)
for all $(\alpha,g),(\beta,h)\in \DD$,
\begin{align*}
&T_{(\alpha,g)} T_{(\alpha,g)}^* T_{(\beta,h)} T_{(\beta,h)}^*
\\&\quad=(t_\alpha,g)(t_\alpha^*,\eta(\alpha)g)(t_\beta,h)(t_\beta^*,\eta(\beta) h)
\\&\quad=(t_\alpha t_\alpha^*,\eta(\alpha)g)(t_\beta t_\beta^*,\eta(\beta)h)
\\&\quad=\begin{cases}(t_\alpha t_\alpha^* t_\beta  t_\beta^*,\eta(\alpha)g)\quad \text{if} \ \eta(\alpha)g=\eta(\beta)h,\\ \ 0 \quad \quad \text{otherwise;}
\end{cases}
\end{align*}
if 
$\eta(\alpha)g\neq\eta(\beta)h$ then $(\alpha,g)\vee (\beta,h)$ is empty, so  we get 
\[ T_{(\alpha,g)} T_{(\alpha,g)}^* T_{(\beta,h)} T_{(\beta,h)}^*= 0 = \bigvee_{(\gamma,k)\in (\alpha,g)\vee (\beta,h)} T_{(\gamma,k)} T_{(\gamma,k)}^*;\]
on the other hand, if 
$\eta(\alpha)g=\eta(\beta)h$ then,
using (\ref{vee-skew}), we get
\begin{align*}
T_{(\alpha,g)} T_{(\alpha,g)}^* T_{(\beta,h)} T_{(\beta,h)}^*&=
(t_\alpha t_\alpha^* t_\beta  t_\beta^*,\eta(\alpha) g)
\\&=
\left(\bigvee_{\delta\in \alpha\vee \beta} t_\delta t_\delta^*, \eta(\alpha)g\right)
\\&=
\bigvee_{\delta\in \alpha\vee \beta} \bigl(t_\delta t_\delta^*,\eta(\alpha)g\bigr)
\\&=
\bigvee_{\delta\in \alpha\vee \beta} \bigl(t_\delta,\eta(\delta)\inv\eta(\alpha) g\bigr)\bigl(t_\delta^*,\eta(\alpha) g\bigr)
\\&=
\bigvee_{\delta \in \alpha\vee \beta} T_{(\delta,\eta(\delta)\inv\eta(\alpha) g)} T_{(\delta,\eta(\delta)\inv\eta(\alpha) g)}^*
\\&=\bigvee_{(\gamma,k)\in (\alpha,g)\vee (\beta,h)} T_{(\gamma,k)} T_{(\gamma,k)}^*,
\end{align*}
and

(4)
if $F' \subseteq (v, g)\DD$ is a finite exhaustive set at $(v,g) \in \DD^0$, then one checks without too much trouble that $F'=\cup_{\alpha \in F} (\alpha, \eta(\alpha)\inv g)$ where $F\subseteq v\CC$ is a finite exhaustive set at $v \in \CC^0$, so we get that
\begin{align*}
T_{(v,g)}
&=(t_v,g)
\\&=\left(\bigvee_{\alpha\in F} t_\alpha t_\alpha^*,g\right)
\\&=\bigvee_{\alpha\in F} (t_\alpha t_\alpha^*,g)
\\&=
\bigvee_{\alpha\in F} (t_\alpha,\eta(\alpha)\inv g)(t_\alpha^*, g)
\\&=
\bigvee_{\alpha\in F} T_{(\alpha,\eta(\alpha)\inv g)} T_{(\alpha,\eta(\alpha)\inv g)}^*
\\&=
\bigvee_{(\alpha,k)\in F'} T_{(\alpha,k)} T_{(\alpha, k)}^*,
\end{align*}
as desired.

We will show that $\theta$ is an isomorphism by constructing an inverse homomorphism
\[
\Theta\:\OO(\CC)\rtimes_\delta G\to \OO(\DD),
\]
and we will get $\Theta$ as the integrated form $\pi\times\phi$ for a covariant representation $(\pi,\phi)$ of the coaction $(\OO(\CC),\delta)$.
Define $Q\:\CC\to M(\OO(\DD))$ by
\[
Q(\alpha)=\sum_{g\in G}t_{(\alpha,g)}.
\]
We must check that the above sum converges strictly in the multiplier algebra $M(\OO(\DD))$ to a partial isometry,
and this follows since if $g\neq h$ the partial isometries $t_{(\alpha,g)}$ and $t_{(\alpha,h)}$ have orthogonal range projections and
orthogonal domain projections.
We will verify that $Q$ is a covariant representation of $\CC$.

(1)
For all $\alpha\in \CC$,
\begin{align*}
Q_\alpha^* Q_\alpha
&=\sum_{g,h\in G}t_{(\alpha,g)}^*t_{(\alpha,h)}
\\&=\sum_{g\in G}t_{(\alpha,g)}^*t_{(\alpha,g)}
\\&=\sum_{g\in G}t_{s(\alpha,g)}
\\&=
\sum_{g\in G}t_{(s(\alpha),g)
}
\\&=Q_{s(\alpha)},
\end{align*}

(2)
if $s(\alpha)=r(\beta)$,
\begin{align*}
Q_\alpha Q_\beta
&=\sum_{g,h\in G}t_{(\alpha,g)}t_{(\beta,h)}
\\&=
\sum_{h\in G}t_{(\alpha,\eta(\beta)h)(\beta,h)}
\\&=
\sum_{h\in G}t_{(\alpha\beta,h)}
\\&=Q_{\alpha\beta},
\end{align*}

(3)
for all $\alpha,\beta\in\CC$,
\begin{align*}
Q_\alpha Q_\alpha^* Q_\beta Q_\beta^*
&=\sum_{g,h,\ell,m\in G}t_{(\alpha,g)}t_{(\alpha,h)}^*t_{(\beta,\ell)}t_{(\beta,m)}^*
\\&=\sum_{g,\ell\in G}t_{(\alpha,g)}t_{(\alpha,g)}^*t_{(\beta,\ell)}t_{(\beta,\ell)}^*
\\&=\sum_{g,\ell\in G}\bigvee_{(\gamma,k)\in (\alpha,g)\vee (\beta,\ell)}
t_{(\gamma,k)}t_{(\gamma,k)}^*
\\&
=\sum_{g\in G}\bigvee_{\delta\in \alpha\vee\beta}t_{(\delta,\eta(\delta)\inv\eta(\alpha)g)}t_{(\delta,\eta(\delta)\inv\eta(\alpha)g)}^*
\\&
\overset{*}=\bigvee_{\delta\in \alpha\vee\beta}\sum_{g\in G}t_{(\delta,\eta(\delta)\inv\eta(\alpha)g)}t_{(\delta,\eta(\delta)\inv\eta(\alpha)g)}^*
\\&=\bigvee_{\delta\in \alpha\vee\beta}\sum_{a\in G}t_{(\delta,a)}t_{(\delta,a)}^*
\\&\overset{**}=\bigvee_{\delta\in \alpha\vee \beta}\sum_{a,b\in G}
t_{(\delta,a)} t_{(\delta,b)}^*
\\&=\bigvee_{\delta\in \alpha\vee \beta}Q_\delta Q_\delta^*,
\end{align*}
where the equality at *
holds 
because 
the range projections of $t_{(\delta,\eta(\delta)\inv\eta(\alpha)g)}$ and $t_{(\delta',\eta(\delta')\inv\eta(\alpha)g)}$
are orthogonal for all $\delta,\delta'\in \alpha\vee \beta$, $\delta\neq \delta'$,
and all the projections commute,
while the equality at ** holds because if $a\neq b$ then the partial isometries 
$t_{(\delta, a)}$ and $t_{(\delta, b)}$ have orthogonal domain projections,
and

(4)
for every finite exhaustive set $F\subseteq v\CC$ at $v\in \CC^0$,
and for each $g\in G$ the set 
$F_g'=\cup_{\alpha \in F} (\alpha, \eta(\alpha)\inv g) \subseteq (v, g)\DD$ is finite exhaustive at $(v,g) \in \DD^0$, so we get that
\begin{align*}
Q_v
&=\sum_{g\in G}t_{(v,g)}
\\&=
\sum_{g\in G}\bigvee_{(\gamma,k)\in F'_g}t_{(\gamma,k)}t_{(\gamma,k)}^*
\\&=
\sum_{g\in G}\bigvee_{\alpha\in F}t_{(\alpha,\eta(\alpha)\inv g)}t_{(\alpha,\eta(\alpha)\inv g)}^*
\\&\overset{*}=
\bigvee_{\alpha\in F}\sum_{g\in G}t_{(\alpha,\eta(\alpha)\inv g)}t_{(\alpha,\eta(\alpha)\inv g)}^*
\\&=\bigvee_{\alpha\in F}\sum_{a\in G}t_{(\alpha,a)}t_{(\alpha,a)}^*
\\&=
\bigvee_{\alpha\in F}\sum_{a,b\in G}t_{(\alpha,a)}t_{(\alpha,b)}^*
\\&=\bigvee_{\alpha\in F}Q_\alpha Q_\alpha^*,
\end{align*}
where the equality at * holds for the same reason as in the preceding verification of property (3).

Thus we can define a homomorphism $\pi\:\OO(\CC)\to \OO(\DD)$
as the integrated form of $Q$.
Moreover, $\pi$ is nondegenerate because
$\pi(\OO(\CC))\OO(\DD)$
includes all the projections
$t_{(v,g)}$ for $(v,g)\in \CC^0\times G=\DD^0$.

Next we show that there is a unique homomorphism $\phi\:c_0(G)\to M(\OO(\DD))$ such that
\begin{equation}\label{phi}
\phi(\Chi_g)=\sum_{v\in \CC^0}t_{(v,g)}\righttext{for}g\in G.
\end{equation}
Since the projections $t_{(v,g)}$ for $g\in G$ are mutually orthogonal,
the above sum converges in the strict topology
to a projection in
$M(\OO(\DD))$.
Since $c_0(G)$ is generated by the pairwise orthogonal projections $\Chi_g$ for $g\in G$,
and since for $g\ne h$
the projections $\phi(\Chi_g)$ and $\phi(\Chi_h)$ are orthogonal,
formula \eqref{phi} uniquely determines a homomorphism $\phi\:c_0(G)\to M(\OO(\DD))$.
Moreover, $\phi$ is nondegenerate because $\phi(c_0(G))\OO(\DD)$
includes all the projections
$t_{(v,g)}$ for $(v,g)\in \CC^0\times G=\DD^0$.

We verify that the pair $(\pi,\phi)$ is a covariant representation of the coaction $(\OO(\CC),\delta)$.
For $\alpha\in\CC$ and $g\in G$,
we have
\begin{align*}
\pi(t_\alpha)\phi(\Chi_g)
&=\sum_{\substack{h\in G\\v\in \CC^0}}t_{(\alpha,h)}t_{(v,g)}.
\end{align*}
In the sum, all terms are 0
except when 
$h=\eta(v)g=g$
and $v=s(\alpha)$,
and the sum gives
\[
\pi(t_\alpha)\phi(\Chi_g)=
t_{(\alpha,g)}.
\]
On the other hand,
\[
\phi(\Chi_{\eta(\alpha)g})\pi(t_\alpha)
=\sum_{\substack{v\in \CC^0\\h\in G}}t_{(v,\eta(\alpha)g)}t_{(\alpha,h)},
\]
in which all terms are 0
except when
$\eta(\alpha)g=\eta(\alpha)h$, \text{i.e.,} $h=g$,
and $v=r(\alpha)$,
and the sum gives
\[
\phi(\Chi_{\eta(\alpha)g})\pi(t_\alpha)
=t_{(\alpha,g)}.
\]
Thus
\[
\pi(t_\alpha)\phi(\Chi_g)=\phi(\Chi_{\eta(\alpha)g})\pi(t_\alpha).
\]
Since the set $\{t_\alpha:\alpha\in\CC\}$ is a set of elements of the spectral subspaces of the coaction $\delta$ that generates the $C^*$-algebra $\OO(\CC)$,
it now follows from \lemref{cov gen lem}
that the pair $(\pi,\phi)$ is a covariant representation of $(\OO(\CC),\delta)$.

Now we check that the integrated form $\Theta=\pi\times\phi$ is an inverse of $\theta$.
First, for all $(\alpha,g)\in \DD$,
\begin{align*}
\Theta\circ\theta(t_{(\alpha,g)})
&=\Theta(t_\alpha,g)
\\&=
(\pi\times\phi)(j_{\OO(\CC)}(t_\alpha)j_G(\Chi_g))
\\&=
\pi(t_\alpha)\phi(\Chi_g)
\\&=t_{(\alpha,g)},
\end{align*}
so $\Theta\circ\theta=\id_{\OO(\DD)}$.

On the other hand, for all $\alpha\in\CC$,
\begin{align*}
\theta\circ\Theta(j_{\OO(\CC)}(t_\alpha))
&=\theta(\pi(t_\alpha))
\\&=\sum_{g\in G}\theta(t_{(\alpha,g)})
\\&=\sum_{g\in G}T_{(\alpha,g)}
\\&=\sum_{g\in G}(t_\alpha,g)
\\&=
\sum_{g\in G}j_{\OO(\CC)}(t_\alpha)j_G(\Chi_g)
\\&=j_{\OO(\CC)}(t_\alpha)
\righttext{(since $\sum_{g\in G}\Chi_g=1$ strictly in $M(c_0(G))$),}
\end{align*}
and for all $g\in G$,
\begin{align*}
\theta\circ\Theta(j_G(\Chi_g))
&=\theta(\phi(\Chi_g))
\\&=\sum_{v\in \CC^0}\theta(t_{(v,g)})
\\&=\sum_{v\in \CC^0}T_{(v,g)}
\\&=\sum_{v\in \CC^0}(t_v,g)
\\&=
\sum_{v\in \CC^0}j_{\OO(\CC)}(t_v)j_G(\Chi_g)
\\&=j_G(\Chi_g)
\righttext{(since $\sum_{v\in \CC^0}t_v=1$ strictly in $M(\OO(\CC))$),}
\end{align*}
so $\theta\circ\Theta=\id_{\OO(\CC)\rtimes_\delta G}$.

Finally, we check the $\gamma-\what\delta$ equivariance.
For $g\in G$, we must show that $\theta\circ\gamma_g=\what\delta_g\circ\theta$,
and
it suffices to check this on a generator $t_{(\alpha,h)}$:
\begin{align*}
(\theta\circ\gamma_g)(t_{(\alpha,h)})
&=\theta(t_{(\alpha,hg\inv)})
\\&=(t_\alpha,hg\inv)
\\&=\what\delta_g(t_\alpha,h)
\\&=(\what\delta_g\circ\theta)(t_{(\alpha,h)}).
\qedhere
\end{align*}
\end{proof}

When $G\act \DD$ is an action of $G$ on a small category $\DD$ by category automorphisms, one can form the semi-direct product $\DD \rtimes G$, which is a finitely aligned LCSC if $\DD$ is. This is not difficult to show directly, but we note that this follows from \cite[Propositions 4.6 and 4.13]{bkqs} by regarding $\DD \rtimes G$ as the Zappa-Sz{\' e}p product of $\DD$ by $G$ with respect to the trivial category cocycle
$\varphi_0\: \DD\times G \to G$, given by $\varphi_0(\alpha, g) = g$ for all $(\alpha, g)\in \DD\times G$.
\begin{cor}\label{dcp}
Assume that $\CC$ is finitely aligned and let $(\CC \times_\eta G)\rtimes G$ denote the finitely aligned LCSC associated with the canonical action $G\act (\CC \times_\eta G)$. Then 
\[\OO((\CC \times_\eta G)\rtimes  G) \simeq \OO(\CC) \otimes \KK(\ell^2(G)).\] 
\end{cor}
\begin{proof} 
It follows from \cite[Remark 5.6 c)]{bkqs} that $\OO((\CC \times_\eta G)\rtimes  G) \simeq \OO(\skpr) \rtimes_\gamma G$. Thus, using Theorem~\ref{maximal}, Theorem~\ref{coaction crossed product isomorphism},
and Katayama crossed-product duality, we get
\[\OO((\CC \times_\eta G)\rtimes  G)  \simeq (\OO(\CC)\rtimes_\delta G) \rtimes_{\what\delta} G \simeq \OO(\CC) \otimes \KK(\ell^2(G)).\qedhere\]
\end{proof} 

\begin{cor}\label{itd}
Assume that $\DD$ is a finitely aligned LCSC and $G$ acts by category automorphisms on $\DD$.
Then there is a cocycle $\eta_0$ on the semi-direct-product $\DD\rtimes G$ given by
\[
\eta_0(\alpha,g)=g,
\]
and
\[
\OO\bigl((\DD\rtimes G)\times_{\eta_0} G\bigr)\simeq
\OO(\DD)\otimes \KK(\ell^2(G)).
\]
\end{cor}

\begin{proof}
The first statement can be verified using routine computations.
Let $\delta_0$ be the coaction of $G$ on $\OO(\DD\rtimes G)$ associated with $\eta_0$, according to \propref{is coaction}.
Also let $\beta$ be the action of $G$ on $\OO(\DD)$
arising from the action $G\curvearrowright \DD$.
Then
it is easy to check that
the isomorphism $\OO(\DD\rtimes G)\simeq \OO(\DD)\rtimes_{\beta} G$
from \cite[Remark~5.6(c)]{bkqs}
is $\delta_0-\what{\beta}$ equivariant.
Thus, 
\thmref{coaction crossed product isomorphism}
and
Imai-Takai crossed-product duality give
\begin{align*}
\OO\bigl((\DD\rtimes G)\times_{\eta_0} G\bigr)
&\simeq \OO(\DD\rtimes G)\rtimes_{\delta_0} G
\\&\simeq \bigl(\OO(\DD)\rtimes_\beta G\bigr)\rtimes_{\what{\beta}} G
\\&\simeq \OO(\DD)\otimes \KK(\ell^2(G)).
\qedhere
\end{align*}
\end{proof}

We also include the following result about gauge actions on $\OO(\CC)$, which is an immediate consequence of Theorem~\ref{coaction crossed product isomorphism} combined with \cite[Appendix A]{enchilada}.

\begin{cor}\label{gauge}
Assume that $\CC$ is finitely aligned and $G$ is  abelian. Let $\what G$ denote the dual group and let $\alpha$ denote the action of $\what G$ on $\OO(\CC)$ corresponding to $\delta$, which is determined by
$\alpha_\gamma(t_\alpha) = \gamma(\eta(\alpha)) t_\alpha$ for $\gamma \in \what G$ and $\alpha \in \CC$.  Then
$\OO(\CC \times_\eta G) \simeq \OO(\CC) \rtimes_\alpha \what G$. 
  \end{cor}

\begin{rem} In the setting of Corollary~\ref{gauge} one may introduce a time evolution $\omega\: \mathbb{R}\to\aut(\OO(\CC))$ 
whenever there exists a continuous homomorphism $x\mapsto \gamma_x$ from $\mathbb{R}$ into $\what G$ by defining $ \omega_x= \alpha_{\gamma_x}$ for each $x\in \mathbb{R}$. In this case we have  
\[\omega_x(t _\alpha) = \gamma_x(\eta(\alpha)) t_\alpha \] 
for all $x \in \mathbb{R}$ and $\alpha \in \CC$. 
Starting with the seminal paper of Laca and Raeburn \cite{LR:affine} on the Toeplitz algebra of the affine semigroup, KMS-states for similar time evolutions on Toeplitz algebras and Cuntz-Krieger algebras of some particular LCSCs have been investigated during the last decade (see e.g.~\cite{HLRS1, HLRS2} where $\CC$ is a higher-rank graph, and \cite{ABLS} where $\CC$ is a right LCM). Thus a study of KMS-states for time evolutions of the type described above for other LCSCs would be natural in the future.     
\end{rem}

\section{Free actions of groups on LCSCs}\label{act}
 Assume that a discrete group $G$ acts on a small category  $\DD$ by automorphisms, and that this action is \emph{free} in the sense that 
for all $g\in G$ and $\lambda \in \DD$,
$g\cdot \lambda = \lambda$ 
implies $g=e$ (the unit of $G$). 

We may then define the quotient small category $\DD/G$ as follows. Setting $[\lambda]= \{g\cdot \lambda : g \in G\}$ for $\lambda \in \DD$, we define
$\DD/G$ as a set by $\DD/G = \big\{[\lambda] : \lambda \in \DD\big\}$. Put $(\DD/G)^0 = \{[v] : v \in \DD^0\}$, and let the range and the source maps $r, s :\DD/G\to (\DD/G)^0$ be given by $r([\lambda])= [r(\lambda)]$ and $ s([\lambda]) = [s(\lambda)] $.

Let $\lambda, \mu \in \DD$ be such that $([\lambda], [\mu])$ is composable in $\DD/G$. 
Then \[[s(\lambda)]=s([\lambda])=r([\mu])= [r(\mu)],\] so there exists $g\in G$ such that $s(\lambda) =g\cdot r(\mu) =  r(g\cdot\mu)$, and $g$ is uniquely determined since the action $G\act\DD$ is free.  
Moreover, if $\lambda' = h\cdot \lambda$ and $\mu' = k\cdot \mu$ for $h,k\in G$, then we have
\[ h\cdot(\lambda (g\cdot\mu)) = (h\cdot\lambda) ((hg)\cdot\mu) = \lambda' ((hgk^{-1})\cdot\mu') =  \lambda' (g'\cdot\mu'),\]
with $g'= hgk^{-1}$.
It follows that the product
\[ [\lambda][\mu] = [\lambda (g\cdot \mu)]\]
is well defined whenever  $s([\lambda])=r([\mu])$, where  $g
\in G$ is determined as above. It is then routine to check that the remaining properties necessary for $\DD/G$ to become a small category hold. 

We note that the quotient map $q\: \DD \to \DD/G$ sending $\lambda$ to $[\lambda]$ is a functor. For if $(\lambda, \mu)$ is a composable pair in $\DD$, then we have \[s([\lambda]) = [s(\lambda)]= [r(\mu)] = r([\mu]),\] so $([\lambda], [\mu])$ is a composable pair in $\DD/G$; as $s(\lambda) = e\cdot r(\mu)$, we get \[ q(\lambda \mu) = [\lambda \mu]=[\lambda (1\cdot \mu)] = [\lambda][\mu] \,.\]
We record this in the following:
\begin{prop} 
Assume that $\DD$ is a small category and $G\act \DD$ is a free action by automorphisms. Then $\DD/G$ is a small category and the quotient map is a functor. Moreover, if $\DD$ is a LCSC, then $\DD/G$ is too.
\end{prop}
\begin{proof} 
It only remains to check the last assertion.
Let $([\lambda],[\mu]), ([\lambda],[\mu'])$ be composable pairs in $\DD/G$ and assume that  $[\lambda][\mu] =  [\lambda][\mu'] $. Then $[\lambda (g\cdot\mu)]= [\lambda (g'\cdot\mu')]$ for some $g, g' \in G$, and this implies that
 \[\lambda (g\cdot\mu)= h\cdot\big(\lambda (g'\cdot\mu')\big) = (h\cdot\lambda) ((hg')\cdot\mu')\]
for some $h\in G$. Thus we have $r(\lambda) = r(h\cdot \lambda) = h\cdot r(\lambda)$, so $h=e$ (by freeness of the $G$-action). This gives that 
\[\lambda (g\cdot\mu)= \lambda (g'\cdot\mu'),\]
hence that $g\cdot\mu = g'\cdot\mu'$ (by left cancellativity of $\DD$). This implies that $[\mu]=[\mu']$, as desired.
\end{proof}  

We will see in Corollary~\ref{D fin align} 
that if $\DD/G$ is finitely aligned, then $\DD$ is too. One may wonder whether the converse holds. 
Here are some observations concerning this problem.

\begin{lem} \label{techn}
Assume that $\DD$ is a LCSC and $G\act \DD$ is a free action by automorphisms, and let $q\: \DD \to \DD/G$ denote the quotient map.

Let $[\lambda], [\mu] \in \DD/G$. Then we have
$[\lambda](\DD/G) \cap [\mu] (\DD/G)\neq \varnothing$ if and only if there exists some $k\in G$ such that $\lambda\DD \cap (k \cdot \mu)\DD \neq \varnothing$, in which case we have \[[\lambda](\DD/G) \cap [\mu] (\DD/G) = \bigcup_{t\in G} q\big( \lambda\DD \cap (t \cdot \mu)\DD \big).\]
\end{lem}

\begin{proof} 
Assume first that $[\lambda](\DD/G) \cap [\mu] (\DD/G)\neq \varnothing$ and let $[\nu] \in [\lambda](\DD/G) \cap [\mu] (\DD/G)$. Then we have
\[[\nu] = [\lambda (g\cdot \gamma)] = [ \mu(h\cdot \delta)]\] for some $\lambda, \gamma , \mu, \delta \in \DD$ and some $g, h \in G$ such that $s(\lambda) = r(g\cdot \gamma)$ and $s(\mu) = r(h\cdot \delta)$.  This implies that there exists $k \in G$ such that 
\[\lambda (g\cdot \gamma) = k\cdot(\mu(h\cdot \delta)) = (k \cdot \mu)(k\cdot (h\cdot \delta)),\]
Thus $\lambda\DD \cap (k \cdot \mu)\DD \neq \varnothing$, and $[\nu] \in q(\lambda\DD \cap (k \cdot \mu)\DD)$. This also shows that  $[\lambda](\DD/G) \cap [\mu] (\DD/G) \subset \bigcup_{t\in G} q\big( \lambda\DD \cap (t \cdot \mu)\DD \big)$.

Conversely, assume that $\lambda\DD \cap (k \cdot \mu)\DD \neq \varnothing$ for some $k\in G$, and let $\alpha = q(\nu) = [\nu]$,  where $\nu \in \lambda\DD \cap (k \cdot \mu)\DD$. Then 
there exist $\gamma', \delta' \in\DD$ such that $s(\lambda) = r(\gamma')$, $s(k\cdot \mu)= r(\delta')$ and 
\[  \nu = \lambda \gamma' = (k\cdot \mu) \delta'.\]
Thus we get that
\[\alpha = [\lambda \gamma'] = [\lambda] [\gamma'] = [ (k\cdot \mu) \delta'] =  [k\cdot \mu] [\delta'] =[\mu] [\delta']. 
\]
Hence, $\alpha \in [\lambda](\DD/G) \cap [\mu] (\DD/G)$, so  $[\lambda](\DD/G) \cap [\mu] (\DD/G)\neq \varnothing$.
This also shows that $\bigcup_{t\in G} q\big( \lambda\DD \cap (t \cdot \mu)\DD \big) \subset [\lambda](\DD/G) \cap [\mu] (\DD/G)$. 
Altogether, this proves the lemma.
\end{proof}

Now,  let  $\DD$ and $G\act \DD$ be as in Lemma~\ref{techn}. Assume that $\DD$ is finitely aligned and $[\lambda](\DD/G) \cap [\mu] (\DD/G)\neq \varnothing$. Set \[K=\{ k \in G :  \lambda\DD \cap (k \cdot \mu)\DD \neq \varnothing\}\] and note that $K$ is nonempty. Moreover, for each $k \in K$ we can find a nonempty finite subset $F_k$ of $ \DD$ such that   
\[\lambda\DD \cap (k \cdot \mu)\DD = \bigcup_{\omega \in F_k} \omega \DD.\] 
Set $F= \bigcup_{k\in K} F_k$. Then we get that 
\begin{align*}
[\lambda](\DD/G) \cap [\mu] (\DD/G) &= \bigcup_{k\in K} q\big( \lambda\DD \cap (k \cdot \mu)\DD \big) \\
&= \bigcup_{k\in K} q\big(  \bigcup_{\omega \in F_k} \omega \DD \big) \\
&= \bigcup_{\omega \in F} [\omega] (\DD/G).
\end{align*}

To be able to conclude that $\DD/G$ is finitely aligned, we would like to show that 
$\bigcup_{\omega \in F} [\omega] (\DD/G) = \bigcup_{[\omega] \in A} [\omega] (\DD/G)$ for some finite subset $A$ of $\DD/G$. 
It is not clear to us that this is true in general.
Anyhow, if $G$ is finite, then $F$ is finite, so the desired equality is obviously true. So we can at least record the following:

\begin{prop}
Assume $\DD$ is a finitely aligned LCSC, and  $G\act \DD$ is a free action of a finite group $G$. Then $\DD/G$ is also a finitely aligned LCSC.
\end{prop}  

 Let now $\CC$ be a small category 
and let $\eta\:\CC\to G$ be a cocycle
into a discrete group $G$.
Recall from Lemma~\ref{act skrew} that $G$ acts on the skew product category $\skpr$ by
\[ g\cdot (\beta, h) = (\beta, hg^{-1}).\]
It is immediate  that this action of $G$ on $\skpr$  is free, and that the resulting quotient category $(\skpr)/G$ is isomorphic to $\CC$, via the (well-defined) map $[(\beta, h)] \mapsto \beta$.  

Conversely, the following analog of the Gross-Tucker theorem for free actions of groups on directed graphs (cp.\ \cite{GT}) holds: 

\begin{thm}[Gross-Tucker] \label{GT-thm} Assume $\DD$ is a small category and $G\act \DD$ is a free action by automorphisms. Then there exists a cocycle $\eta\: \DD/G \to G$ such that $\DD$  is equivariantly isomorphic to  $(\DD/G)\times_\eta G $.
\end{thm}

\begin{proof}
We expand the proof briefly sketched by Kumjian and Pask in the case where $\DD$ is a higher rank graph 
(cp.\ \cite[Remark 5.6]{kp:kgraph}). 

First, for each $w \in (\DD/G)^0$, choose $v_w \in \DD^0$ such that $[v_w] = w$. 
Next, let $\alpha \in \DD/G$, so $\alpha= [\lambda] =\{g\cdot \lambda: g\in G\}$ for some $\lambda \in \DD$. Since $G$ acts freely on $\DD$, the map $g\mapsto g\cdot \lambda$ is a bijection from $G$ onto $\alpha$. Moreover,  since \[[v_{r(\alpha)}] = r(\alpha) =r([\lambda])= [r(\lambda)], 
\] 
there is a unique $g \in G$
such that 
$v_{r(\alpha)} = g\cdot r(\lambda)$, i.e., such that $v_{r(\alpha)} = r(g\cdot \lambda)$. 
Thus, we get that $\lambda_\alpha := g \cdot \lambda$ is the unique element of $\DD$ such that  \[[\lambda_\alpha] = \alpha\quad \text{ and} \quad r(\lambda_\alpha) = v_{r(\alpha)}.\]
Since 
$[s(\lambda_\alpha)] = s(\alpha) = [v_{s(\alpha)}]$, 
we also have that there is a unique element $\eta(\alpha) \in G$ such that \[s(\lambda_\alpha) = \eta(\alpha) \cdot v_{s(\alpha)}.\]
Assume now that $\beta \in \DD/G$, and $r(\beta) = s(\alpha)$. Note that
\[r(\eta(\alpha)\cdot \lambda_\beta) = \eta(\alpha) \cdot r(\lambda_\beta)  = \eta(\alpha) \cdot v_{r(\beta)} = \eta(\alpha) \cdot v_{s(\alpha)} = s(\lambda_\alpha).\]
Thus, the pair $(\lambda_\alpha, \eta(\alpha)\cdot\lambda_\beta)$ is composable in $\DD$. Since \[ [\lambda_\alpha (\eta(\alpha)\cdot\lambda_\beta)] = [\lambda_\alpha] [\eta(\alpha)\cdot\lambda_\beta] = \alpha [\lambda_\beta] = \alpha \beta\]
and 
\[ r\big(\lambda_\alpha (\eta(\alpha)\cdot\lambda_\beta)\big) = r(\lambda_\alpha) = v_{r(\alpha)} = v_{r(\alpha\beta)} = r(\lambda_{\alpha\beta})\]
we get that  \[\lambda_{\alpha\beta} = \lambda_\alpha (\eta(\alpha)\cdot\lambda_\beta).\] This gives that 
\begin{align*} \eta(\alpha\beta)\cdot v_{s(\beta)} &= \eta(\alpha\beta)\cdot v_{s(\alpha\beta)} \\ &= s(\lambda_{\alpha\beta})\\
 &= s(\lambda_{\alpha}(\eta(\alpha)\cdot \lambda_\beta)) \\
&= s(\eta(\alpha)\cdot \lambda_\beta)\\ 
&= \eta(\alpha)\cdot s(\lambda_\beta) \\ 
&= \eta(\alpha)\cdot(\eta(\beta)\cdot v_{s(\beta)}) \\
&= (\eta(\alpha)\eta(\beta))\cdot v_{s(\beta)},
\end{align*}
hence that $\eta(\alpha\beta) = \eta(\alpha)\eta(\beta)$, by freeness of the $G$-action. This shows that the map $\eta\: \DD/G\to G$ is a cocycle.

We may then define $\rho\: (\DD/G) \times_\eta G \to \DD$ by $\rho(\alpha, g) = 
(\eta(\alpha)g)^{-1} \cdot \lambda_\alpha$.
 It is routine to check that $\rho$ is an equivariant isomorphism. As a sample, we show that $\rho$ is a functor. Let $\big((\alpha, g), (\beta, h)\big)$ be a composable pair in $(\DD/G) \times_\eta G$. Then we have that 
$g=\eta(\beta)h$,
so we get
\begin{align*} 
\rho\big((\alpha, g) (\beta, h)\big) &= 
\rho( \alpha\beta, h)
\\&= 
(\eta(\alpha\beta)h)^{-1} \cdot \lambda_{\alpha\beta} 
\\&= 
(\eta(\alpha)g)^{-1} \cdot \big(\lambda_{\alpha}(\eta(\alpha)\cdot\lambda_{\beta})\big) 
\\ &=
\big((\eta(\alpha)g)^{-1} \cdot \lambda_{\alpha}\big)\big(((\eta(\alpha)g)^{-1}\eta(\alpha))\cdot\lambda_{\beta}\big)
\\&= 
\big((\eta(\alpha)g)^{-1} \cdot \lambda_{\alpha}\big)\big((\eta(\beta)h)^{-1}\cdot\lambda_{\beta}\big) 
\\ &= \rho(\alpha, g)\rho(\beta, h).
\qedhere
\end{align*}
\end{proof}

 \begin{cor} \label{D fin align} Assume $\DD$ is a LCSC, $G\act \DD$ is a free action by automorphisms, and $\DD/G$ is finitely aligned. Then $\DD$ is finitely aligned. 
Moreover, there exists a coaction $\delta$ of $G$ on $\OO(\DD/G)$ such that \[ \OO(\DD) \simeq \OO(\DD/G)\rtimes_\delta G.\]
Also, letting  $\beta$ denote the natural action of $G$ on $\OO(\DD)$ arising from the action $G\act \DD$, we have that
 \[\OO(\DD) \rtimes_\beta G \simeq \OO(\DD/G) \otimes \KK(\ell^2(G))\,.\]
  \end{cor}
\begin{proof} The first assertion follows from Theorem~\ref{GT-thm} and \lemref{skew finitely aligned}.
Next, Theorem~\ref{GT-thm} gives us a cocycle $\eta\: \DD/G \to G$ such that, letting 
$\delta$ denote the associated coaction of $G$ on $\OO(\DD/G)$, and using  Theorem~\ref{coaction crossed product isomorphism}, we get  
\[ \OO(\DD) \simeq  \OO((\DD/G) \times_\eta G) \simeq \OO(\DD/G) \rtimes_\delta G.\]
Also, letting
$\gamma$ denote the induced action of $G$ on $\OO((\DD/G) \times_\eta G)$, and using Corollary~\ref{dcp}, we get
\[\OO(\DD) \rtimes_\beta G \simeq \OO((\DD/G) \times_\eta G) \rtimes_\gamma G \simeq \OO(\DD/G) \otimes \KK(\ell^2(G))\,.
\qedhere\]
\end{proof}

\section{Invariant cocycles and Zappa-Sz\'ep products}

Let $(\CC, H, \varphi)$ be a category system in the sense of \cite[Definition~4.1]{bkqs}. Thus, $\CC$ is a small category, $H$ is a group acting on the set $\CC$ by permutations in such way that
\[r(h\alpha)=hr(\alpha)\midtext{and}s(h\alpha)=hs(\alpha)\righttext{for all}h\in H,\alpha\in \CC,\]
and $\varphi\:H\times \CC\to H$ is a category cocycle for this action. Further, let $\CC\rtimes^\varphi H$ denote the Zappa-Sz\'ep product of $(\CC,H,\varphi)$.

Assume that $\CC$ is a finitely aligned LCSC. Then $\CC\rtimes^\varphi H$ is also a finitely aligned LCSC, cp.\ \cite[Propositions 4.6 and 4.13]{bkqs}. Moreover, assume that $\psi\:\CC \to G$ is  a cocycle into a group $G$ that is \emph{$H$-invariant} in the sense that $\psi(h\alpha) = \psi(\alpha) $ for all $h\in H, \alpha\in \CC$. 

We can then promote $\psi$ to a cocycle 
$\eta_\psi\:\CC\rtimes^\varphi H \to G$ by setting
\[\eta_\psi(\alpha, h) = \psi(\alpha) \righttext{for all} h\in H,\alpha\in \CC.\]
Indeed, for any $(\alpha,h),(\beta,h')\in \CC\rtimes^\varphi H$ with $s(\alpha,h)=r(\beta,h')$, that is, such that $h\inv s(\alpha)=r(\beta)$, we  have
\begin{align*}
\eta_\psi\big((\alpha,h)(\beta,h')\big)&=\eta_\psi\bigl(\alpha(h\beta),\varphi(h,\beta)h'\bigr)=\psi(\alpha(h\beta))\\
&= \psi(\alpha)\psi(\beta) = \eta_\psi(\alpha, h) \eta_\psi(\beta, h').
\end{align*}

Under the same assumptions, it is easy to check that we can also define a category system $(\CC\times_\psi G, H,\wilde\varphi)$ by setting $ h(\alpha, g) = (h\alpha, g)$ and $\wilde{\varphi}(h, (\alpha, g)) = \varphi(\alpha, g)$
for all $h\in H$ and $(\alpha, g)\in \CC\times_\psi G$. Then it is routine to verify that the map $\big((\alpha, h), g\big) \mapsto \big((\alpha, g), h\big)$ gives an isomorphism from $(\CC\rtimes^\varphi H)\times_{\eta_\psi} G$ onto $(\CC\times_\psi G)\rtimes^{\wilde\varphi} H$.

{As a special case, let us consider an Exel-Pardo system $(E, H, \varphi)$, cp.\ \cite{exelpardo, bkqexelpardo}, and the associated category system $(E^*, H, \varphi)$, cp.\ \cite{bkqs}.
So,
$E=(E^0, E^1,r,s)$ is a directed graph, $\varphi$ is  a graph cocycle for an action of $H$ on $E$, and $E^*$ is the category of finite paths in $E$. Since $E^*$ is a finitely aligned LCSC (in fact, it is singly aligned), this fits within the above setting. 
Let $f\:E^1\to G$ be a map into a group $G$. Then $f$ induces a cocycle $\psi_f\:E^*\to G$ in the obvious way. Assume that $f$ is constant on each orbit of the action of $H$ on $E^1$, 
so $f$ may be considered as a function $f\:E^1/H \to G$.
Then $\psi_f$ is easily seen to be invariant under the associated action of $H$  on $E^*$ (cp.\ \cite[Example 4.4]{bkqs}). Thus, in this case, we obtain a $G$-valued cocycle $\eta_f:=\eta_{\psi_f}$ on $ E^*\rtimes^\varphi H$, given by $\eta_f(v,h) = {e}_G$ (the unit of $G$) and
$\eta_f(\alpha, h) = f(e_1)\cdots f(e_n)$ for all $v\in E^0, \alpha=e_1\cdots e_n \in E^*\setminus E^0$ and $h\in H$. If we for example choose $G=\Z$ and let $f\:E^1\to \Z$ be the function  constantly equal to 1, we get that $\eta_f(\alpha, h) = \psi_f(\alpha) = \ell(\alpha)$ (i.e., the length of $\alpha$) for all $\alpha \in E^*$ and $h \in H$. 

Applying our previous work to the above set-up, we get the following result:
\begin{prop}\label{ep-coaction}
Let  $(E, H, \varphi)$ be an Exel-Pardo system, $G$ a group, and  $f\: E^1/H \to G$ a map. Then $f$ induces a 
$G$-valued cocycle $\eta_f$ on the Zappa-Sz{\' e}p product category $E^*\rtimes^\varphi H$, which gives rise to a coaction $\delta_f$ of $G$ on $ \OO(E^*\rtimes^\varphi H)$
such that 
\[ \OO(E^*\rtimes^\varphi H)\rtimes_{\delta_f} G \simeq \OO((E^*\rtimes^\varphi H)\times_{\eta_f} G).\] 
\end{prop}
This result makes it possible to write  $ \OO(E^*\rtimes^\varphi H)$ as the $C^*$-algebra of a Fell bundle over $G$ for different choices of $G$ and $f$, which might be useful to study properties $ \OO(E^*\rtimes^\varphi H)$. We remark that if $E$ is row-finite, then \cite[Corollary 6.8]{bkqs} gives that  $ \OO(E^*\rtimes^\varphi H)$ is isomorphic to the Exel-Pardo algebra associated to $(E, H, \varphi)$.  

Specializing further, let us assume for simplicity that $E$ is a finite graph on $N$ vertices with no sources (although the last condition is not really necessary). Ordering the vertices from $1$ to $N$, the edge set $E^1$ is determined by an  $N\times N$ adjacency matrix $A=[a_{i,j}]$ with entries in $\N\cup\{0\}$ having no zero columns. Then any $N\times N$ matrix $B$ with entries in $\Z$ determines an Exel-Pardo system $(E, \Z, \varphi_B)$ (where $\Z$ fixes all vertices of $E$), cp.\ \cite{exelpardo, bkqexelpardo}, thus giving rise to the (singly aligned) LCSC \[\mathcal{K}_{A,B}:=E^*\rtimes^{\varphi_B} \Z.\] Combining \cite{exelpardo}, \cite{bkqexelpardo} and \cite{bkqs}, one deduces  
that $\OO(\mathcal{K}_{A,B})$ is isomorphic to the Katsura algebra $\OO_{A,B}$ introduced in \cite{KatKirchberg}. Hence the set-up described above enables one to define $G$-valued cocycles of the type $\eta_f$ on  $\mathcal{K}_{A,B}$, hence coactions of $G$ on $\OO_{A,B}$, for any group $G$. As the set $E^1_{i,j}$ of edges in $E$ from $j$ to $i$, which consists of $a_{i,j}$ elements, is left invariant by the action of $\Z$, we may for example pick an element $g_{i,j} \in G$ for every $(i,j) \in N\times N$ satisfying that $a_{i,j} \neq 0$ and define 
$f:E^1\to G$ by $f(e) = g_{i,j}$ whenever $e\in E^1_{i,j}$.  
We note that, more generally, we could have handled in a similar way any countable row-finite graph $E$ having no sources, as in \cite{KatKirchberg}.

\section{On the universal group of a small category}

Let  $\CC$ denote a  small category. A natural question is whether $\CC$ has a \emph{universal group} $(U,j)$, in the sense that $j\: \CC\to U$ is a cocycle into a group $U$ such that for every cocycle $\psi$ from $\CC$ into a group $H$
there exists a unique group homomorphism $\psi'$ making the diagram
\[
\xymatrix{
\CC \ar[r]^-j \ar[dr]_\psi
&U \ar@{-->}[d]^{\psi'}_{!}
\\
&H
}
\]
commute. This question 
may be answered positively by applying a more general construction due to Higgins
\cite[Proposition~19, p.~65]{higgins}.
We provide here 
a self-contained proof of this result.

\begin{prop} \label{uni group}
Every small category $\CC$ has a universal group $(U, j)$. 
\end{prop}
\begin{proof}
 Let $\bar\CC$ be a set disjoint from $\CC$ having the same cardinality as $\CC$, and fix a bijection $\CC\to \bar\CC$. We write the image of $\alpha \in \CC$ in $\bar\CC$ as $\bar\alpha$. Let $\F(\bar\CC)$ denote the free group on the set $\bar \CC$, and let $R \subset \F(\bar\CC)$  be given by
\[R=\big\{ \bar\alpha\,\bar\beta\,(\bar{\alpha\beta})\inv : \alpha, \beta \in \CC, s(\alpha)=r(\beta)\big\}.\]
Let then $U$ be the group with presentation $\langle \bar\CC\, | \,R\rangle$, that is, \[U= \F(\bar\CC)/N(R),\] where $N(R)$ denotes the normal closure of $R$ in $\F(\bar\CC)$. For each $f\in \F(\bar\CC)$, we write $[f]$ for the coset $fN(R) \in U$, so that $f\mapsto [f]$ gives the canonical homomorphism from $\F(\bar\CC)$ onto $U$. We may then define a map $j\:\CC\to U$ by $j(\alpha) =[\bar\alpha]$ for each $\alpha\in \CC$. For each $\alpha, \beta \in \CC$ such that $ s(\alpha)=r(\beta)$ we have
\[ j(\alpha\beta)= [\bar{\alpha\beta}] = [\bar{\alpha}\,\bar{\beta}] =  [\bar{\alpha}]\,[\bar{\beta}]= j(\alpha)j(\beta)\,,\]
so $j$ is a cocycle.

Consider now  a cocycle $\psi$ from $\CC$ into a group $H$. 
By the universal property of $\F(\bar\CC)$, there is  a unique homomorphism $\varphi:\F(\bar\CC)\to H$ satisfying 
$\varphi(\bar\alpha) = \psi(\alpha)$ for each $\alpha \in \CC$. For $\alpha, \beta\in \CC$ such that $s(\alpha)=r(\beta)$ we then have
\begin{align*} \varphi\big(\bar\alpha\,\bar\beta\,(\bar{\alpha\beta})\inv\big)&= \varphi(\bar\alpha)\,\varphi(\bar\beta)\,\varphi(\bar{\alpha\beta})\inv\\
&= \psi(\alpha)\,\psi(\beta)\,\psi({\alpha\beta})\inv\\
&= \psi(\alpha)\,\psi(\beta)\,\psi(\beta)\inv\psi(\alpha)\inv\\ &= e \text{ (the unit of } H).
\end{align*}
Thus $R$ is contained in the normal subgroup $\ker(\varphi)$, and it follows that $N(R) \subset \ker(\varphi)$.
Hence, the map $\psi'\: U\to H$ given by 
\[ \psi'\big([f]\big) = \varphi(f) \quad \text{for all }\, f \in \F(\bar\CC)\]
is a well-defined homomorphism. Moreover, we have
\[(\psi'\circ j)(\alpha)=   \psi'([\bar\alpha]) = \varphi(\bar\alpha) = \psi(\alpha) \]
for every $\alpha \in \CC$, as desired. Finally,  assume $\phi$ is also a homomorphism from $U$ into $H$ such that $\psi = \phi\circ j$. Then for every $\alpha \in \CC$ we have 
\[ \phi(j(\alpha))= \psi([\bar\alpha]) =  \psi'(j(\alpha)).\]
As the range of $j$ generates $U$ as a group, we get that $\phi=\psi'$.  
\end{proof}

\begin{rem} By universality, it follows readily that if $(U,j)$ and $(U', j')$ are universal groups for $\CC$, then they are uniquely isomorphic, in the sense that there is a unique isomorphism $\pi\:U\to U'$ 
making the following diagram commute:
\[
\xymatrix{
\CC \ar[r]^-j \ar[dr]_{j'}
&U \ar[d]^\pi_\simeq
\\
&U'
}
\]
\end{rem}

When we want to display the attachment of a universal group $U$ to a small category $\CC$ we write $U(\CC)$. We also sometimes write  $j_\CC$ instead of $j$ for the canonical cocycle from $\CC$ into $U(\CC)$. 

\begin{rem} If  $\DD$ is a small category and $\Phi: \CC \to \DD$ is a functor, then $j_\DD \circ \Phi$ is a cocycle from $\CC$ into $U(\DD)$, so there is a unique homomorphism $\Phi_{*}: U(\CC) \to U(\DD)$ such that $\Phi_{*}\circ j_\CC= j_\DD\circ \Phi$. 
\end{rem}

\begin{rem}
V.~Gould and M.~Kambites
\cite[Theorem~2.3]{GouKam}
have shown
that there exists a faithful functor $\sigma_\CC$ from $\CC$ into a left cancellative monoid  $M(\CC)$ such that $(M(\CC), \sigma_\CC)$ is universal.
  By considering $U(\CC)$ as a monoid we get that there exists a unique homomorphism $\ell: M(\CC) \to U(\CC)$ such that $j_\CC = \sigma_\CC\circ \ell$.
\end{rem}

\begin{rem}
Assume that $\CC$ is left cancellative. Recall from \cite{bkqs} that we write ${\rm ZM}(\CC)$ for the inverse semigroup with zero consisting of zigzag maps associated with $\CC$. Let $\tau\: \CC \to {\rm ZM}(\CC)$ denote the canonical map, i.e., $\tau_\alpha\: s(\alpha)\CC \to \alpha\CC$ is given by $\tau_\alpha(\beta) = \alpha\beta$ for each $\alpha \in \CC$. As observed by S.~Margolis
(see \cite{MilSte}),
every inverse semigroup $S$ with zero has a universal group $\mathcal{U}(S)$.
To state the universal property of this group,  recall that a partial homomorphism from $S$ to a group $G$ is a map $\varphi\: S\setminus \{0\} \to G$ such that $\varphi(st)= \varphi(s)\varphi(t)$ whenever $st \neq 0$.  Then the identity map on $S$ induces a partial homomorphism $\iota$ from $S$ to $\mathcal{U}(S)$, and any partial homomorphism from $S$ into a group factors through $\iota$ via a group homomorphism. Applying this when $S = {\rm ZM}(\CC)$, we get a cocycle  $\eta\: \CC \to \mathcal{U}({\rm ZM}(\CC))$ by setting $\eta = \iota \circ \tau$. Proposition~\ref{uni group} then gives a unique homomorphism $\eta': U(\CC) \to \mathcal{U}({\rm ZM}(\CC))$ such that $\eta = \eta'\circ j_\CC$, i.e., 
making the diagram 
\[
\xymatrix{
\CC \ar[r]^-{j_{\CC}} \ar[d]_\tau \ar[dr]^\eta
&U(\CC) \ar@{-->}[d]^{\eta'}
\\
{\rm ZM}(\CC) \ar[r]^-\iota & \mathcal{U}({\rm ZM}(\CC))
}
\]
commute.
\end{rem}

We now discuss the relationships among the universal group of a small category $\CC$,  the fundamental groupoid of $\CC$, and the fundamental group of $\CC$.
We first give a short review of the latter two concepts.
 Much of this is in \cite[Appendix~A]{BH}
and \cite{pqr:cover}. The development in the latter paper is for higher-rank graphs,
but carries over to the setting of small categories by deleting everything to do with the degree functors. The main motivation in \cite[Appendix~A]{BH} and \cite{pqr:cover} comes from the theory of coverings.

Let 
$\GG$ be a groupoid,
and $i\:\CC\to\GG$ a functor.
The pair $(\GG,i)$ is called a \emph{fundamental groupoid}
of $\CC$ if
for every functor $\phi$ from $\CC$ to a groupoid $\HH$
there exists a unique groupoid morphism $\phi'$ making the following diagram commute:
\[
\xymatrix{
\CC \ar[r]^-i \ar[dr]_\phi
&\GG \ar@{-->}[d]^{\phi'}_{!}
\\
&\HH
}
\]
The  functor $i\:\CC\to\GG$ takes $\CC^0$ bijectively onto $\GG^0$,
and we identify them, so that $\CC^0=\GG^0$.
It follows as usual from universality that any two fundamental groupoids $(\GG,i)$ and $(\GG',i')$ of $\CC$ are
\emph{uniquely isomorphic} in the sense that there is a unique groupoid isomorphism  making the diagram
\[
\xymatrix{
\CC \ar[r]^-i \ar[dr]_{i'}
&\GG \ar[d]^\simeq
\\
&\GG'
}
\]
commute.

Every small category $\CC$ has a fundamental groupoid $(\GG(\CC), i_\CC)$:  a construction may be found for example in \cite[Section 19.1]{schubert}, where
$\GG(\CC)=\CC[\CC\inv]$ is the groupoid obtained as the category of fractions associated to $\CC$.
Alternatively, the reader may consult \cite[III.\ CC.A]{BH}, or \cite[Section II.3.1]{DDGKM},
where the fundamental groupoid of $\CC$ is called the \emph{enveloping groupoid}.
We stress that, in general, a  small category $\CC$ does not embed in $\GG(\CC)$. Of course, $\CC$ would have to be cancellative, but this is not enough.
It is also not enough for $\CC$ to be a higher-rank graph (see \cite[Section~7]{pqr:groupoid}), although the embedding does hold for graphs of rank 1, i.e., directed graphs.
Some sufficient Ore-type conditions for this to hold are given in \cite[Section II.3.2]{DDGKM}. 

The assignment $\CC \mapsto i_\CC$
gives a natural transformation from the identity functor to the functor $\CC\mapsto \GG(\CC)$,
in the strong sense that
for any functor $\phi\:\CC\to\DD$ between small categories
the map $\phi_*$ is the \emph{unique} groupoid morphism
making the diagram
\[
\xymatrix{
\CC \ar[d]_\phi \ar[r]^-{i_\CC}
&\GG(\CC) \ar[d]^{\phi_*}
\\
\DD \ar[r]_-{i_\DD}
&\GG(\DD)
}
\]
commute.

We next recall that  $\CC$ is said to be \emph{connected} if the equivalence relation generated by $\{(v, w) \in \CC^0\times\CC^0 : v\CC w \neq \varnothing\}$ is $\CC^0\times\CC^0$.
This happens if and only if $\GG(\CC)$ is connected, i.e., if $v\GG(\CC)w \neq \varnothing$ for all $v, w\in \CC^0$. 

Assume that $\CC$ is connected.
Then the isotropy group
\[
\pi(\CC,v):=v\GG(\CC)v
\]
is called the \emph{fundamental group of  $\CC$ at  $v\in \CC^0$}. For any $v,w$ in $\CC^0$
the groups $\pi(\CC,v)$ and $\pi(\CC,w)$ are isomorphic, so any of these groups gives the \emph{fundamental group $\pi(\CC)$ of $\CC$}. Moreover, if $v \in \CC^0$ then there exists a cocycle $\eta$ from $\GG(\CC)$ into $\pi(\CC, v)$. Indeed, for each $w \in \CC^0$, we can pick 
$t_w \in w\GG(\CC)v$, 
and  set \[\eta(a) = t_{r(a)}\inv a t_{s(a)}\] for each $a \in \GG(\CC)$. The map $\eta\circ i_\CC \: \CC\to \pi(\CC, v)$ is then a cocycle on $\CC$. 
It can be shown that  the canonical projection from the skew product $\CC\times_{\eta} \pi(\CC,v)$ onto $\CC$ gives a  universal covering of $\CC$ (cp.\ \cite{pqr:cover}).

\begin{rem}
Since $U(\CC)$ is a groupoid, there is a unique cocycle
$k: \mathcal{G}(\CC)\to U(\CC)$ such that $j = k\circ i_\CC$. More generally, every cocycle $\psi\:\CC\to H$ into a group $H$ factors uniquely through a cocycle $\phi\:\GG(\CC) \to H$.
It follows that $U(\CC)$ is isomorphic to $U(\GG(\CC))$. Thus, to study universal groups of small categories we might as well consider groupoids, as we do in the sequel.
\end{rem}

Let $\GG$ be a groupoid. For each $x\in \GG^0$ let $\GG_x=x\GG x$ denote the isotropy group of $\GG$ at $x$.
We first assume that $\GG$ is connected.
Then it is well-known that $\GG$ is isomorphic as a groupoid to the product  
$(\GG^0\times \GG^0)\times \GG_x$
for any $x\in\GG^0$, where $\GG^0\times \GG^0$ denotes the groupoid associated to the full equivalence relation on $\GG^0$.
Indeed, begin by fixing $x\in\GG^0$,
and for each $y\in\GG^0$ fix $t_y\in y\GG x$,
with $t_x=x$.
Then the map
\[
\theta\:\GG\to (\GG^0\times \GG^0)\times \GG_x
\]
defined for any $a\in z\GG y$ by
\[
\theta(a)=\bigl((z,y),t_z\inv a t_y\bigr)
\]
is an isomorphism, with inverse given by $\theta\inv((z, y), h) = t_zht_y\inv$.

In the above discussion, we used a convenient type of subset of $\GG$ that we will need again, so we pause to give it a name:

\begin{defn}\label{max tree}
Let $\GG$ be a connected groupoid, and pick $x\in\GG^0$.
For each $y\in \GG^0$ choose $t_y\in y\GG x$, with $t_x=x$.
We call $T=\{t_y\}_{x\in\GG^0}$ a \emph{maximal tree} in $\GG$ \emph{rooted at $x$}.
\end{defn}

\begin{prop}\label{universal group}
Let $\GG$ be a connected groupoid.
Fix $x\in \GG^0$,
and let $T=\{t_y\}_{y\in\GG^0}$ be a maximal tree in $\GG$ rooted at $x$.
Let $H=\GG_x$ be the isotropy group of $\GG$ at $x$.
Let $S=\GG^0\minus \{x\}$, and define
\[
j\:\GG\to \F(S)*H
\]
by
\[
j(a)=\bar{z}(t_z\inv at_y)\bar{y}\,\inv\righttext{for}a\in z\GG y.
\]
Then $(\F(S)*H,j)$ is a universal group of $\GG$.
\end{prop}

\begin{proof}
First we verify that $j$ is a cocycle:
for $a\in z\GG y,b\in y\GG u$,
\begin{align*}
j(a)j(b)
&=\bar{z}(t_z\inv at_y)\bar{y}\,\inv
\bar{y}(t_y\inv bt_u)\bar{u}\,\inv
\\&=\bar{z}(t_z\inv abt_u)\bar{u}\,\inv
\\&=j(ab).
\end{align*}

We now verify the universal property.
Let $\psi\:\GG\to K$ be any cocycle.
To define a group homomorphism
\[
\psi'\:\F(S)*H\to K,
\]
we need to choose elements $\{k_y\}_{y\in S}$ in $K$
and a homomorphism $\rho\:H\to K$.
We choose
\[
k_y=\psi(t_y)\midtext{and}\rho=\psi|_H,
\]
and let $\psi'\:\F(S)*H\to K$ be the associated homomorphism.
Then for $a\in z\GG y$ we have
\begin{align*}
(\psi'\circ j)(a)
&=\psi'\left(\bar{z}(t_z\inv at_y)\bar{y}\,\inv\right)
\\&=k_z\,\rho\left(t_z\inv at_y\right)k_y\inv
\\&=\psi(t_z)\psi\left(t_z\inv at_y\right)\psi(t_y)\inv
\\&=\psi\left(t_zt_z\inv at_yt_y\,\inv\right)
\\&=\psi(a).
\qedhere
\end{align*}
\end{proof}

\begin{cor} \label{uni group connected cat}
Assume that $\CC$ is a connected small category. Pick any $v\in \CC^0$ and set $S=\CC^0\minus{\{v\}}$. Then $U(\CC) \simeq \F(S) * \pi(\CC)$.
\end{cor}

A maximal connected subgroupoid of a groupoid $\GG$ is called a \emph{component} of $\GG$. Clearly, $\GG$ is then a disjoint union of its components. Letting $\{\GG_\lambda\}_{\lambda \in \Lambda}$ denote a partition of $\GG$ into its components, we get \[U(\GG) \simeq \star_{\lambda \in \Lambda} U(\GG_\lambda).\] 
The proof of this fact is quite standard, and goes as follows. First, for each $\lambda$, let  $j_\lambda\:\GG_\lambda \to U(\GG_\lambda)$ denote the canonical functor, and  identify $U(\GG_\lambda)$ with its canonical image in the free product $G=\star_{\lambda \in \Lambda} U(\GG_\lambda)$. Then we get a functor $j\:\GG\to G$ by defining $j(a) = j_\lambda(a)$ whenever $a \in \GG_\lambda$. Next, assume that $\phi$ is a cocycle from $\GG$ into some group $H$. For each $\lambda$, let $\phi_\lambda$ denote the restriction of $\phi$ to $\GG_\lambda$, and let $\phi'_\lambda\:U(\GG_\lambda) \to H $ denote the homomorphism associated to this cocycle.  Using the universal property of $G$, we get that
there exists a unique homomorphism $\phi'\:G\to H$ such that $\phi'$ restricts to $\phi'_\lambda$ on $U(\GG_\lambda)$ for each $\lambda$. By construction, we have that $\phi' \circ j = \phi$. Hence, we have shown that $(G,j)$ is a universal group of $\GG$, as desired.

Combining this result with Proposition~\ref{universal group},
we get the following description of the universal group of a groupoid.

\begin{thm} Let $\GG$ be a groupoid and let $\{\GG_\lambda\}_{\lambda \in \Lambda}$ denote a partition of $\GG$ into its components. For each $\lambda\in \Lambda$, pick $x_\lambda \in \GG_\lambda^0$ and  set $S_\lambda = \GG_\lambda^0 \minus \{x_\lambda\}$. Then  $U(\GG)$  is isomorphic to the free product 
\[\star_{\lambda \in \Lambda} \big(\F(S_\lambda) * \pi(\GG_\lambda)\big).\]
\end{thm}
It follows that the universal group of a small category $\CC$ is nontrivial whenever the fundamental groupoid of $\CC$ is nontrivial.

\section{Coverings and connectedness of skew products}\label{sec:cover}

We will use the theory of coverings of small categories to help decide when skew products are connected.
Much of the following is in \cite[Appendix~A]{BH}
and \cite{pqr:cover}. The development in the latter paper is for $k$-graphs,
but carries over to the setting of small categories by deleting everything to do with the degree functors.
Thus, many of the results are just a routine reformulation of the similar results from \cite{pqr:cover}, 
and we therefore skip their proofs.

\begin{defn}
A \emph{covering} of a small category $\CC$ is a 
surjective functor
$p\:\DD\to\CC$,
where $\DD$ is a small category,
such that for every $v\in \DD^0$,
the restrictions
\[
\DD v\to \CC p(v)\midtext{and}v\DD\to p(v)\CC
\]
are bijective.
If 
$q\:\EE\to\CC$ is another covering,
a \emph{morphism} from $(\DD,p)$ to $(\EE,q)$ is a functor
$\phi\:\DD\to\EE$ making the diagram
\[
\xymatrix{
\DD \ar[rr]^-\phi \ar[dr]_p
&&\EE \ar[dl]^q
\\
&\CC
}
\]
commute.
A covering $p\:\DD\to\CC$ is \emph{connected} if $\DD$ (and hence $\CC$) is connected.
\end{defn}

The class of all coverings of a small category $\CC$ is a category with the above morphisms.

\begin{defn}
An \emph{action} of  small category $\CC$ on a set $V$
is a functor
from $\CC$ to the category of sets such that
\begin{enumerate}
\item
$V$ is the disjoint union of the sets $V_x$ associated to the vertices $x$ of $\CC$;

\item
for every $\alpha\in y\CC x$,
the associated map
$v\mapsto \alpha v$ from $V_x$ to $V_y$
is bijective.
\end{enumerate}
We write the action as $\CC\act V$.
\end{defn}

The above definition reduces to the familiar notion of groupoid action when $\CC$ is a groupoid.
Much of the theory of groupoid actions carries over to category actions.
For example:

\begin{prop}\label{tfm cat}
Let $\CC\act V$ be an action of a small category $\CC$.
Then the set
\[
\CC*V=\{(\alpha,v): \alpha\in\CC, v\in V_{s(\alpha)}\}
\]
becomes a small category with vertex set
\[
(\CC*V)^0=\CC^0*V\subset \CC*V
\]
and operations
\begin{itemize}
\item
$s(\alpha,v)=(s(\alpha),v)$

\item
$r(\alpha,v)=(r(\alpha),\alpha v)$

\item
$(\alpha,\beta w)(\beta,w)=(\alpha\beta,w)$.
\end{itemize}
\end{prop}

\begin{proof}
As mentioned at the start of the proof of \cite[Proposition~3.3]{pqr:cover} in the case of $k$-graphs, this is a completely routine adaptation of the special case where $\CC$ is a groupoid.
\end{proof}

Note that if $\CC$ is a groupoid, then so is $\CC*V$.

\begin{lem}\label{actions}
Let $\CC$ be a small category with fundamental groupoid $(\GG,i)$.
For every action $\CC\act V$ of  $\CC$ on a set $V$
there is a unique action $\GG\act V$ such that
\[
i(\alpha)v=\alpha v\righttext{for all}(\alpha,v)\in \CC*V.
\]
Moreover, every action of $\GG$ arises in this way from a unique action of $\CC$.
\end{lem}

\begin{proof}
An action of $\CC$ is a functor from $\CC$ into the groupoid of bijections among the sets $\{V_x\}_{x\in \CC^0}$. By universality, this functor factors uniquely through the canonical functor $i\:\CC\to \GG$,
giving an action $\GG\act V$.
In the other direction, any action of $\GG$,
when composed with $i$,
trivially gives an action of $\CC$, from which the groupoid action can be recovered by the above argument.
\end{proof}

The preceding lemma gives a bijection between actions of $\CC$ and of $\GG(\CC)$.

\begin{prop}\cite[Theorem~3.7]{pqr:cover}\label{gpd cover prop}
Let $p\:\DD\to\CC$ be a covering,
and let 
$(\HH,j)$ and $(\GG,i)$
be the fundamental groupoids of 
$\DD$ and $\CC$, 
respectively.
Then
the associated groupoid morphism $p_*\:\HH\to \GG$ is also a covering.
Equivalently, the map
\[
(p_*,s)\:\HH\to \GG*\DD^0
\]
is a groupoid isomorphism.
\end{prop}

\begin{proof}
The proof of the second statement is a routine modification of the proof of \cite[Theorem~3.7]{pqr:cover},
in which
the roles of $\CC$ and $\DD$ are played by $k$-graphs $\Lambda$ and $\Omega$, respectively,
and the modifications are as follows:
\begin{itemize}
\item
the ``skeletons'' $E$ and $F$ of $\Lambda$ and $\Omega$
should now be the underlying graphs of $\CC$ and $\DD$,
and

\item
the sets of ``commuting squares'' $S_\Lambda$ and $S_\Omega$
are to be replaced by the sets of relations $K_\CC$ and $K_\DD$
determined by the composition operations of $\CC$ and $\DD$, respectively,
e.g.,
\[
K_\CC=\{(\alpha\beta,\gamma):\alpha,\beta,\gamma\in \PP(E),
\alpha\beta=\gamma\text{ in }\CC\}.
\]
\end{itemize}
Then the crux of the argument is a careful analysis of the commuting diagram (see the proof of \cite[Theorem~3.7]{pqr:cover} for undefined notation)
\[
\xymatrix{
\PP(F^+) \ar[rrr]^-R \ar[ddd]_{q_*} \ar@/_4pc/[dddd]_{(q_*,s)}^\simeq
&&&\HH \ar[ddd]^{p_*} \ar@/^4pc/[dddd]^{(p_*,s)}
\\
&F^+ \ar[d]_q \ar@{_(->}[ul]
&F \ar[d]^p \ar@{_(->}[l] \ar[ur]^j
\\
&E^+ \ar@{_(->}[dl]
&E \ar@{_(->}[l] \ar[dr]^i
\\
\PP(E^+) \ar[rrr]_-Q
&&&\GG
\\
\PP(E^+)*\DD^0 \ar[rrr]_{Q*\id} \ar[u]_{\pi_1}
&&&\GG*\DD^0 \ar[u]^{\pi_1}
}
\]

To see that the second statement implies the first, note that the isomorphism $\HH\simeq \GG*\DD^0$ takes the morphism $p_*$ to the
coordinate projection $p_\GG$.
It follows that $p_*$ maps $\HH v$ injectively onto $\GG x$,
and so is a groupoid covering.
Conversely, the first statement implies the second
because if we take $\HH=\GG*\DD^0$
then $p_*$ is the coordinate projection $p_\GG$ given by
\[
p_\GG(a,v)=a\righttext{for}(a,v)\in \GG*\DD^0.
\qedhere
\]
\end{proof}

A
groupoid action $\GG\act V$ on a set $V$ is called \emph{transitive}
if $\GG v=V$ for some --- and hence every ---  $v\in V$.

\begin{cor}
Let $p\:\DD\to\CC$ be a covering,
and let $(\HH,j)$ and $(\GG,i)$ be the fundamental groupoids of $\DD$ and  $\CC$, respectively.
Then $\DD$ 
is connected if and only if the corresponding groupoid action $\GG\act \DD^0$ is transitive.
\end{cor}

\begin{proof}
The category $\DD$ is connected if and only if $\HH$, equivalently $\GG*\DD^0$, is.
For $(a,v)\in \GG*\DD^0$, we have $s(a,v)=(s(a),v)$ and $r(a,v)=(r(a),av)$.
It follows that $\HH$ is connected if and only if for all $u,v\in \DD^0$ there exists $a\in\GG$ such that $u=av$, i.e., if and only if $\GG$ acts transitively on $\DD^0$.
\end{proof}

\begin{cor}[{\cite[Theorem~2.5]{pqr:cover}}]\label{normalizer}
let $p\:\DD\to\CC$ be a connected covering,
$x\in \CC^0$, and $v\in p\inv(x)$.
Then the normalizer $N$ of $p_*(\pi(\CC,v))$ in $\pi(\CC,x)$
acts on the covering $(\DD,p)$ by automorphisms,
and in fact
\[
\aut(\DD,p)\simeq N/p_*\pi(\DD,v).
\]
\end{cor}

\begin{cor}[{\cite[Corollary 3.9]{pqr:cover}}]
Let $p\:\DD\to\CC$ be a covering,
and let $x\in \CC^0$ and $v\in p\inv(x)$.
Then $p_*$ maps the fundamental group $\pi(\DD,v)$ isomorphically onto the stability group
$S_v$ of the corresponding groupoid action $\GG(\CC)\act \DD^0$.
\end{cor}

\begin{proof}
Since $p_*$ is a covering, it maps $\pi(\DD,v)$
isomorphically onto \emph{some} subgroup of $\pi(\GG,x)$.
For $c\in\HH v$ and $a=p_*(c)$
we have
$av=r(c)$,
so $c\in\pi(\DD,v)$ if and only if $a\in S_v$. The result follows.
\end{proof}

\subsection*{Classification of transitive groupoid actions}

For the ease of the reader, we collect here some results from \cite{pqr:cover}.

\begin{prop}[{\cite[Proposition~4.1]{pqr:cover}}]
Let $\GG\act V$ be a transitive groupoid action on a set $V$ and $x\in \GG^0$.
Then the family $\{S_v:v\in V_x\}$ is a conjugacy class of subgroups of $\GG_x$.
\end{prop}

\begin{prop}[{\cite[Proposition~4.2]{pqr:cover}}]
Let a groupoid $\GG$ act transitively on both $V$ and $U$,
and let $x\in \GG^0$, $v\in V_x$, and $u\in U_x$.
Then there is a morphism
$(\GG\act V)\to (\GG\act U)$ taking $v$ to $u$
if and only if $S_v\subset S_u$.
\end{prop}

\begin{prop}[{\cite[Proposition~4.3]{pqr:cover}}]\label{action normalizer}
Let $\GG\act V$ be a transitive groupoid action, and let $x\in \GG^0$ and $v\in V_x$.
Then the normalizer $N(S_v)$ of $S_v$ in $\GG_x$
acts on the right of the action $\GG\act V$ by automorphisms, and
this gives rise to an isomorphism
\[
\aut(\GG\act V)\simeq N(S_v)/S_v.
\]
\end{prop}

If $\eta\:\GG\to G$ is a cocycle into a group $G$,
the associated \emph{cocycle action} $\GG\act (\GG^0\times G)$ is given by
\[
a\bigl(s(a),g\bigr)=\bigl(r(a),\eta(a)g\bigr).
\]
We write $\GG^0\times_\eta G$ to indicate the set $\GG^0\times G$ equipped with the cocycle action.

\begin{prop}[{\cite[Proposition~4.5]{pqr:cover}}]
Let $\GG$ be a connected groupoid and $x\in \GG^0$.
There is a cocycle $\eta\:\GG\to \GG_x$ such that the associated
action $\GG\act (\GG^0\times_\eta \GG_x)$
is free and transitive.
\end{prop}

Note that if a group $G$ acts on the right of a groupoid action $\GG\act V$ by automorphisms,
then $\GG$ acts on the orbit space $V/G$ by
\[
\alpha\cdot (vG)=(\alpha\cdot v)G.
\]

\begin{prop}[{\cite[Proposition~4.6]{pqr:cover}}]
Let $\GG\act V$ be a free and transitive groupoid action,
$x\in \GG^0$, $v\in V_x$,
and $H$ a subgroup of $\GG_x$.
Let $H\curvearrowright (\GG\act V)$ as in \propref{action normalizer}.
Then the associated action $\GG\act V/H$ is transitive, and $H=S_{vH}$.
\end{prop}

\subsection*{Universal coverings of small categories}

\begin{defn}
A covering $q\:\EE\to\CC$ of small categories is \emph{universal}
if it is connected and for every connected covering $p\:\DD\to\CC$ there is a
unique
morphism $\phi\:(\EE,q)\to (\DD,p)$.
\end{defn}

\begin{thm}[{\cite[Theorem~2.7]{pqr:cover}}]
Every connected small category $\CC$ has a universal covering,
any two universal coverings are uniquely
isomorphic,
and a connected covering $q\:\EE\to\CC$ is universal if and only if
$\pi(\EE,v)=\{v\}$ for some \(and hence every\) $v\in \EE^0$.
\end{thm}

\begin{thm}[{\cite[Theorem~2.8]{pqr:cover}}]
Let $q\:\EE\to\CC$ be a universal covering,
$x\in \CC^0$, $v\in q\inv(x)$,
and $H$ a subgroup of $\pi(\CC,x)$.
Let 
$H$ act on the covering $(\EE,q)$
as in \thmref{normalizer}.
Then the associated covering $p\:\EE/H\to \CC$
is connected, and
\[
H=p_*(\pi(\EE/H,vH).
\]
Moreover, every connected covering of $\CC$ is isomorphic to one of the above coverings $\EE/H\to\CC$.
\end{thm}

\begin{thm}[{\cite[Corollary 5.5]{pqr:cover}}]
If $p\:\DD\to\CC$ is a connected covering,
then the following are equivalent:
\begin{enumerate}
\item $p\:\DD\to\CC$ is universal;
\item the corresponding action $\GG(\CC)\curvearrowright \DD^0$ is free;
\item $\pi(\DD,v)=\{v\}$ for some \(and hence every\) $v\in \DD^0$.
\end{enumerate}
\end{thm}

\begin{thm}[{\cite[Corollary 6.5]{pqr:cover}}]
Let $\CC$ be a connected small category and let $x\in \CC^0$.
Then there is a cocycle $\eta\:\CC\to \pi(\CC,x)$ such that the associated skew-product covering
\[
\CC\times_\eta \pi(\CC,x)\to \CC
\]
is universal.
\end{thm}

\begin{prop}[{\cite[Proposition~6.6]{pqr:cover}}]
Let $\CC$ be a small category and $\eta\:\CC\to G$ a cocycle.
Then $G$ acts freely on the right of the skew-product covering $\CC\times_\eta G\to \CC$ via
\[
(\alpha,g)h=(\alpha,gh)\righttext{for}\alpha\in\CC,g,h\in G.
\]
\end{prop}

\subsection*{Connected skew products}
We are interested in when skew products of cocycles on groupoids, or more generally (in fact, in some sense equivalently) on small categories, are connected. We will use the following concept.
\begin{defn}
Let $\GG$ be a groupoid and $\HH\subset\GG$ a subgroupoid.
A \emph{retraction} of $\GG$ on $\HH$ is a morphism $\phi\:\GG\to \HH$ that is the identity map on $\HH$.
\end{defn}

\begin{lem}
Let $\GG$ be a connected groupoid,
$x\in\GG^0$,
and $T=\{t_y\}_{y\in\GG^0}$ a maximal tree rooted at $x$.
Define a cocycle $\eta\:\GG\to\GG_x$ by
\[
\eta(a)=t_z\inv at_y\righttext{if}a\in z\GG y.
\]
Then $\eta$ is a retraction whose kernel contains $T$.
Moreover, every retraction $\GG\to\GG_x$ is of this form for a unique maximal tree rooted at $x$.
\end{lem}

\begin{proof}
The first statement is routine.
For the second, suppose that $\eta\:\GG\to\GG_x$ is a retraction.
For each $y\in\GG^0$,
choose any $a\in y\GG x$,
and define
\[
t_y=a\eta(a)\inv.
\]
Because $\eta$ is a retraction,
$t_y$ is independent of the choice of $a$,
and is in $\ker\eta$.
Moreover, $t_x=x$, so $T=\{t_y\}_{y\in\GG^0}$ is a maximal tree rooted at $x$.
A trivial computation shows that $\eta(a)=t_z\inv at_y$ for $a\in z\GG y$.
It follows quickly from the 
definition of retraction that $\eta$ maps each set $y\GG x$ bijectively onto $\GG_x$,
so the intersection $y\GG x\cap \ker\eta$ contains only $t_y$, proving uniqueness of $T$.
\end{proof}

\begin{rem}
If $\psi\:\GG\to G$ is a cocycle on a groupoid $\GG$, one obvious necessary condition for $\GG\times_\psi G$ to be connected is that $\GG$ be connected,
since the coordinate projection
\[
p_\GG\:\GG\times_\psi G\to \GG
\]
is a covering.
\end{rem}

\begin{defn}
A cocycle $\psi\:\CC\to G$ on a small category $\CC$ is \emph{nondegenerate} if $\psi(\CC)$ generates $G$ as a group.
\end{defn}

\begin{thm}\label{connected equivalent}
Let $\GG$ be a connected groupoid and $\psi\:\GG\to G$ a nondegenerate cocycle.
The following are equivalent:
\begin{enumerate}
\item[(i)]
$\GG\times_\psi G$ is connected;

\item[(ii)]
$\psi(\GG_x)=G$ for every $x\in \GG^0$;

\item[(iii)]
$\psi(\GG_x)=G$ for some $x\in \GG^0$;

\item[(iv)]
there is a maximal tree of $\GG$ contained in $\ker\psi$;

\item[(v)]
for every $x\in\GG^0$ there is a maximal tree of $\GG$ rooted at $x$ contained in $\ker\psi$;

\item[(vi)]
$\psi$ factors through a retraction $\eta\:\GG\to \GG_x$ for every $x\in \GG^0$;

\item[(vii)]
$\psi$ factors through a retraction $\eta\:\GG\to \GG_x$ for some  $x\in \GG^0$.
\end{enumerate}
\end{thm}

Note that when we say ``$\psi$ factors through a retraction $\eta\:\GG\to \GG_x$'' we mean that there is a group homomorphism $\phi\:\GG_x\to G$ such that $\psi=\phi\circ\eta$,
but in fact this is equivalent to the simpler property $\psi=\psi\circ\eta$ since $\eta$ is a retraction.

\begin{proof}
Assume (i).
Then the associated action $\GG\act (\GG^0\times G)$ is transitive.
Pick $x\in\GG^0$.
Then the restricted action $\GG_x\act (\{x\}\times G)$ is transitive,
which implies that
\[
G=\psi(\GG_x)e=\psi(\GG_x).
\]
Thus  (ii) holds. Of course  (ii) trivially implies  (iii).

Assume  (iii),
and choose $x\in\GG^0$ such that $\psi(\GG_x)=G$.
We will show how to define a maximal tree $T=\{t_y\}_{y\in\GG^0}$ rooted at $x$
and contained in $\ker\psi$.
Of course it suffices to consider $y\ne x$.
Since
$\GG$ is connected we can choose $a\in y\GG x$,
and then by assumption there exists $b\in \GG_x$ such that $\psi(a)=\psi(b)$,
and then $ab\inv\in y\GG x$ and $\psi(ab\inv)=e$,
so we can take $t_y=ab\inv$.
Thus  (iv) holds.

Now assume  (iv),
and choose a maximal tree $T=\{t_y\}_{y\in\GG^0}$ rooted at a vertex $z$ and contained in $\ker\psi$.
Then for any $x\in\GG^0$, the set $S=\{t_yt_x\inv\}_{y\in\GG^0}$ is a maximal tree rooted at $x$
contained in $\ker\psi$.

Assume  (v), and let $x\in\GG^0$.
Choose a maximal tree $T=\{t_y\}_{y\in\GG^0}$ rooted at $x$ and contained in $\ker\psi$.
Let $\eta\:\GG\to \GG_x$ be the associated retraction.
Then for $a\in z\GG x$ we have
\[
\psi(a)=\psi(t_z)\psi(\eta(a))\psi(t_y)=\psi\circ\eta(a).
\]
Thus  (vi) holds, and of course  (vi) trivially implies (vii).

Finally, assume (vii),
so that $\psi$ factors through a retraction $\eta\:\GG\to\GG_x$ with associated maximal tree $T=\{t_y\}_{y\in\GG^0}$ contained in $\ker\psi$.
Since $\psi$ is nondegenerate,
$G$ is generated as a group by the image $\psi(\GG)=\psi(\GG_x)$,
so in fact $G=\psi(\GG_x)$ because $\psi|_{\GG_x}$ is a group homomorphism.
But then 
\[
G=\psi(\GG)e,
\]
so the associated action $\GG\act (\GG^0\times G)$ is transitive,
and hence the skew product $\GG\times_\psi G$ is connected.
\end{proof}

\begin{ex}
Note that the image of a groupoid $\GG$ under a cocycle might not be a subgroup.
For example,
let $\GG=\{x,y,a,a\inv\}$ be the full equivalence relation on the set $\{x,y\}$,
with $s(a)=x$.
Then the map $\psi\:\GG\to\Z$ defined by
\[
\psi(a)=1\midtext{and}\psi(a\inv)=-1
\]
(and of course $\psi(x)=\psi(y)=0$) is a cocycle
whose image is the subset $\{0,1,-1\}$ of $\Z$.
\end{ex}

\begin{ex}\label{use universal group}
By \thmref{connected equivalent}, if we are interested in connected skew products, we must assume that $\GG$ is connected and the cocycle $\psi\:\GG\to G$ is surjective.
However, this is \emph{not} enough.
Consider the connected groupoid $\GG=\{0,1\}^2\times \Z$,
where $\{0,1\}^2$ denotes the full equivalence relation on $\{0,1\}$.
For this discussion it will be convenient to instead describe $\GG$ as follows:
let $\GG^0=\{x,y\}$, let
$\GG_x=\<k\>$ be an infinite cyclic group with generator $k$,
and choose $a\in y\GG x$.

Then we can take the universal group $(U,j)$ of $\GG$ to be given by
\[
U=\F_1*\GG_x=\F_2=\<b,c\>,
\]
with
\[
j(a)=b
\midtext{and}
j(k)=c.
\]

Let $G$ be the infinite dihedral group
\[
\<\theta,d:\theta^2=1,\theta d\theta=d\inv\>.
\]
We define a cocycle $\psi\:\GG\to G$ by giving the associated homomorphism
$\psi'\:\F_2\to G$,
and by freeness it suffices to define $\psi'$ on the generators $b,c$:
\[
\psi'(b)=\theta,\quad \psi'(c)=d.
\]
Then the image of the cocycle $\psi=\psi'\circ j$ contains
\[
\{d^n:n\in\Z\}\cup \{\theta d^n:n\in\Z\}=G,
\]
so $\psi$ is surjective.
On the other hand, $\psi(\GG_x)=\<d\>\ne G$,
so by \thmref{connected equivalent}
the skew-product groupoid
$\GG\times_\psi G$ is not connected.
\end{ex}
\begin{rem}
This example illustrates how the universal group of a category may provide valuable information. 
\end{rem}

\begin{rem}
In \thmref{connected equivalent}, the hypothesis that the cocycle $\psi$ be nondegenerate is necessary for the implication (iv) implies (i). Indeed, if we start with any connected skew product and then just enlarge the target group $G$, (iv) (and hence (v)--(vii) also) will still hold but (i) will not.
Of the properties (iv)--(vii), (vii) seems to us to be the most interesting.
\end{rem}

We generalize \thmref{connected equivalent} to connected small categories:

\begin{cor}\label{connected equivalent category}
Let $\CC$ be a connected small category and $\eta\:\CC\to G$ a nondegenerate cocycle.
Let $\psi\:\GG(\CC)\to G$ denote the associated cocycle. The following are equivalent:
\begin{enumerate}
\item[(i)]
$\CC\times_\eta G$ is connected;

\item[(ii)]
the cocycle $\psi$
maps $\pi(\CC,x)$ onto $G$ for every $x\in \CC^0$;

\item[(iii)]
the cocycle $\psi$ maps $\pi(\CC,x)$ onto $G$ for some $x\in \CC^0$.
\end{enumerate}
\end{cor}

\begin{proof}
This follows 
from \thmref{connected equivalent}, because by \propref{gpd cover prop}
with $\DD=\CC\times_\eta G$
we have
\[
\GG(\CC\times_\eta G)\simeq \GG(\CC)*(\CC\times_\eta G)^0=\GG(\CC)*(\CC^0\times G),
\]
and it is quite easy to check that we have an isomorphism
\begin{align*}
\GG(\CC)*(\CC^0\times G)&\simeq \GG(\CC)\times_\psi G
\\
\bigl(a,(s(a),g)\bigr)&\mapsto (a,g),
\\ 
\end{align*}
and $\CC\times_\eta G$ is connected if and only if its fundamental groupoid is.
\end{proof}

In the above corollary, the equivalence (i)\iff (iii) is a version of \cite[Corollary~5.6]{pask-rho} (which deals with the special case where $\CC$ is the path category of a directed graph).




\end{document}